\documentclass{article}
\usepackage[utf8]{inputenc}
\usepackage[english]{babel}

\newcommand{\PA}[1]{{\color{blue}{\bf PA:} #1}}

\usepackage{bbold}
\usepackage{amsmath}
\usepackage{amsfonts}
\usepackage{amssymb}
\usepackage{graphicx}
\usepackage{amsthm}
\usepackage{stmaryrd}
\usepackage{enumitem}
\usepackage{nameref}

\usepackage{bbm}
\usepackage{amsfonts}
\usepackage{verbatim}
\usepackage{mathtools}
\usepackage{mathrsfs}
\usepackage{multirow}
\usepackage{array}
\usepackage{subcaption}
\usepackage{todonotes}
\usepackage[linesnumbered,ruled,vlined]{algorithm2e}
\usepackage{cite}
\usepackage{breqn}
\usepackage{thm-restate}
\usepackage{hyperref}
\usepackage{cleveref}
\usepackage{fullpage}
\usepackage{authblk}

\usepackage{tikz}
\usetikzlibrary{decorations.pathreplacing}
\usetikzlibrary{calc}

\tikzstyle{vertex}=[circle,draw, minimum size=8pt, scale=1, inner sep=0.5pt]
\tikzstyle{arc}=[->, very thick]
\tikzstyle{edge}=[very thick]

\newcommand{\NN}{\ensuremath{\mathbb{N}}}
\newcommand{\ZZ}{\ensuremath{\mathbb{Z}}}
\newcommand{\RR}{\ensuremath{\mathbb{R}}}

\newtheorem{thm}{Theorem}[section]
\newtheorem{prop}[thm]{Proposition}

\newtheorem{lem}[thm]{Lemma}
\newtheorem{corol}[thm]{Corollary}
\newtheorem{Def}[thm]{Definition}

\newtheorem{Rq}[thm]{Remark}

\newtheorem{claim}{Claim}[thm]

\newenvironment{proofclaim}[1][]{\par\noindent {\it Proof of claim}.\ }{\hfill$\lozenge$\par\vspace{11pt}}

\newcommand{\bigcupl}{\bigcup\limits}

\newcommand{\bigsqcupl}{\bigsqcup\limits}
\newcommand{\suml}{\sum\limits}

\newcommand{\infl}{\inf\limits}

\newcommand{\minl}{\min\limits}

\newcommand{\limsupl}{\limsup\limits}
\newcommand{\liminfl}{\liminf\limits}

\newcommand{\argmin}{\text{argmin}}
\newcommand{\argminl}[1]{\underset{#1}{\argmin}\ }
\newcommand{\argmax}{\text{argmax}}
\newcommand{\argmaxl}[1]{\underset{#1}{\argmax}\ }

\newcommand{\bigfloor}[1]{\left\lfloor #1 \right\rfloor}

\newcommand{\bigceil}[1]{\left\lceil #1 \right\rceil}
\newcommand{\abs}[1]{|#1|}

\newcommand{\intZ}[1]{\llbracket #1 \rrbracket}

\newcommand{\tend}[1]{\underset{#1}{\longrightarrow}}

\newcommand{\dic}{\vec{\chi}}
\newcommand{\bid}{\overset \leftrightarrow}

\renewcommand{\epsilon}{\varepsilon}
\newcommand{\eps}{\varepsilon}

\newcommand{\m}{\mathcal}

\usepackage{tikz}
\usetikzlibrary[graphs, arrows, backgrounds, intersections, positioning, fit, petri, calc, shapes, decorations.pathmorphing]
\usetikzlibrary{arrows.meta}

\title{Various bounds on the minimum number of arcs in a $k$-dicritical digraph}
\author[1]{Pierre Aboulker}
\author[1]{Quentin Vermande}

\affil[1]{DIENS, \'Ecole normale sup\'erieure, CNRS, PSL University, Paris, France.}

\begin{document}

\maketitle

\begin{abstract}
The dichromatic number $\vec{\chi}(G)$ of a digraph $G$ is the least integer $k$ such that $G$ can be partitioned into $k$ acyclic digraphs. A digraph is $k$-dicritical if $\vec{\chi}(G) = k$ and each proper subgraph $H$ of $G$ satisfies $\vec{\chi}(H) \leq k-1$. 

We prove various bounds on the minimum number of arcs in a $k$-dicritical digraph, a structural result on $k$-dicritical digraphs and a result on list-dicolouring. 
We characterise $3$-dicritical digraphs $G$ with $(k-1)\abs{V(G)} + 1$ arcs.  
For $k \geq 4$, we characterise $k$-dicritical digraphs $G$ on at least $k+1$ vertices and with  $(k-1)\abs{V(G)} + k-3$ arcs, generalising a result of Dirac. 
We prove that, for $k \geq 5$, every $k$-dicritical digraph $G$ has at least $(k-\frac 1 2 - \frac 1 {k-1}) \abs{V(G)} - k(\frac 1 2 - \frac 1 {k-1})$ arcs, which is the best known lower bound. 
We  prove that the number of connected components induced by the vertices of degree $2(k-1)$ of a $k$-dicritical digraph is at most the number of connected components in the rest of the digraph, generalising a result of Stiebitz.  Finally, we generalise a Theorem of Thomassen on list-chromatic number of undirected graphs to list-dichromatic number of digraphs. 
\end{abstract}
\pagebreak 
\tableofcontents

\section{Introduction and results}
\label{sec:intro}

A \emph{colouring} of a directed graph (shortly digraph) $G$ is a partition of the set of vertices of $G$ into independent subsets and the \emph{chromatic number} $\chi(G)$ of $G$ is the minimum size of such a partition. This is a very natural generalisation of the notion of colouring of graphs, but not a very suitable one since it does not take into account the orientation of the arcs. Neumann-Lara introduced in 1982~\cite{Neu82} the notion of dicolouring of digraphs, which is another natural generalisation of the concept of colouring of graphs. It is more suitable than the previous one since it takes into account the orientation of the arcs. A \emph{dicolouring} of a digraph $G$ is a partition of the set of vertices of $G$ inducing acyclic digraphs, and the dichromatic number $\dic(G)$ of $G$ is the minimum size of such a partition. This is indeed a generalisation as, with the correspondence between graphs and \emph{symmetric digraphs} (that is digraphs obtained from undirected graphs by replacing each edge by a digon, where a \emph{digon} is a pair of anti-parallel arcs), 
we have, for every symmetric digraph $G$, $\chi(G) = \dic(G)$. 
\medskip 

We study minimal obstructions to dicolourability. A digraph $G$ is \emph{dicritical} if, for every proper subdigraph $H$ of $G$, $\dic(H) < \dic(G)$. 
We also say that $G$ is \emph{$k$-dicritical} when $G$ is dicritical and $\dic(G) = k$. 
Observe that any digraph $G$ contains a $\dic(G)$-dicritical subdigraph. 
This means that many problems on the dichromatic number of digraphs reduce to problem on dicritical digraphs, whose structure is more restricted. 
We are interested in their sparsity: we aim at computing the minimum number of arcs in a $k$-dicritical digraph on $n$ vertices. Lemma~\ref{lem:ord} shows that this value is well defined for $n \ge k \ge 2$.
\medskip

It is well known that every vertex in a $k$-critical (undirected) graph has degree at least $k-1$, and hence a $k$-critical graph $G$ has at least $\frac{1}{2}(k-1)\abs{V(G)}$ edges. Brooks' theorem implies a simple characterisation of graphs $G$ with exactly $\frac{1}{2}(k-1)\abs{V(G)}$ edges. 

\begin{thm}[\cite{Bro41}]
Let $G$ be a connected graph. Then $\chi(G) \le \Delta(G)+1$ and equality holds if and only if $G$ is an odd cycle or a complete graph.
\end{thm}

Similarly, it is well known (see Lemma~\ref{lem:pouet}(\ref{lem:d})) that every vertex in a $k$-dicritical digraph has degree at least $2(k-1)$ and hence a $k$-dicritical digraph has at least $(k-1)\abs{V(G)}$ arcs. Brooks' theorem was generalised in~\cite{Mohar10} (see also~\cite{AA21}) to digraphs, and implies a simple characterisation of the $k$-dicritical digraphs $G$ with exactly $(k-1)\abs{V(G)}$ arcs. 
For $G$ a digraph, let $\Delta_{max}(G)$ be the maximum over the vertices of $G$ of the maximum of their in-degree and their out-degree.

\begin{thm}[Theorem~2.3 in~\cite{Mohar10}]\label{thm:brooks} Let $G$ be a connected digraph. Then $\dic(G) \le \Delta_{max}(G)+1$ and equality holds if and only if $G$ is a directed cycle, a symmetric cycle of odd length or a symmetric complete digraph on at least $4$ vertices.
\end{thm}

In 1957, Dirac went one step further and proved the following.

\begin{thm}[\cite{Dirac57}]\label{thm_nono:dirac1}
Let $k \geq 4$ and $G$ a $k$-critical graph. If $G$ is not $K_k$, then \[2\abs{E(G)} \ge (k-1)\abs{V(G)}+k-3.\]

\end{thm}

We generalise this theorem to digraphs:
\begin{thm} \label{thm_intro:directed_dirac_1} 
Let $k \geq 4$ and $G$ a $k$-dicritical digraph. If $G$ is not $\bid K_k$, then \[\abs{A(G)} \ge (k-1) \abs{V(G)} + k-3.\]
\end{thm}

Dirac later identified the graphs for which the bound is tight (whose set we denote $\m D_k$, see Section~\ref{subsec:Refined_Dirac} for a definition) and improved his bound. Recall that the Kronecker symbol $\delta_{i,j}$ is equal to $1$ if $i=j$ and $0$ otherwise. 

\begin{thm}[\cite{Dirac74}]\label{thm_nono:dirac2}
Let $k \ge 4$ and let $G$ a $k$-critical graph. If $G$ is neither $K_k$ nor in $\m D_k$, then \[2\abs{E(G)} \ge (k-1) \abs{V(G)} + (k-1-\delta_{k, 4}).\]
\end{thm}

It turns out that our bound is also tight exactly for the digraphs in $\m D_k$ (via the identification between graphs and symmetric digraphs):

\begin{thm}\label{thm_intro:directed_dirac_2} 
Let $k \ge 4$ and $G$ be a $k$-dicritical digraph. If $G$ is neither $\bid K_k$ nor $\m D_k$, then:\[\abs{A(G)} \ge (k-1)\abs{V(G)} + (k-2).\]
\end{thm}
The perspicacious reader will notice that our bound is weaker than Dirac's when $k \geq 5$. Yet our bound is tight for some digraphs (which are thus not symmetric, see Section~\ref{subsec:Refined_Dirac}). 

It is well known that the only $3$-critical graphs are odd cycles, which is the reason why Dirac's two mentioned results deal with $k \geq 4$. 
However, $3$-dicritical digraphs are not as simple, as witnessed by the fact that deciding if a digraph is $2$-dicolourable is $NP$-complete~\cite{Bokal2004}. 
We prove the following, where $\mathcal D'_3$ is a class of $3$-dicritical digraphs defined in Section~\ref{subsec:dirack=3}: 

\begin{thm}\label{thm_intro:directed_dirac_k=3}
Let $G$ be a $3$-dicritical digraph. If $G$ is not a symmetric cycle of odd length, then 
\[\abs{A(G)} = 2\abs{V(G)} + 1\]
if and only if $G \in \mathcal D'_3$, and otherwise 
\[\abs{A(G)} \geq 2\abs{V(G)} + 2\]
\end{thm}
\bigskip 

Gallai was the first~\cite{Gal63b} to find a lower bound with a better slope than $\frac{1}{2}(k-1)$. His result was improved by Krivelevich~\cite{Kriv97} using Gallai's method together with a result of Stiebitz~\cite{Stiebitz82} that we were able to generalise to digraphs: 

\begin{thm}\label{thm_intro:nbcompconnex} 
Let $k \ge 3$, $G$ a $k$-dicritical digraph such that $S = \{x \in G, d(x) \le 2(k-1)\}$. Then the number of connected components of $G-S$ is at most the number of connected components of $G[S]$. 

\end{thm}

Gallai's method works on digraphs, but we obtained better bounds through other means.

In the undirected case, Kostochka and Yancey~\cite{KoYa14} obtained a closed form for the minimum number of edges of a $k$-critical graph on $n$ vertices in an infinite set of cases:
\begin{thm}[Theorem~4 in~\cite{KoYa14}] \label{thm:KY} 

\[\abs{E(G)} \ge \bigceil{\frac{(k+1)(k-2)\abs{V(G)}-k(k-3)}{2(k-1)}}.\] 
This bound is exact for $k = 4$ and $n \ge 6$ and for $k \ge 5$ and $n \equiv 1$ (mod $k-1$).
\end{thm}
Unfortunately we were not able to obtain a comparable result. 
Still, adapting their method, we were able to get the following, which is the best known lower bound on the minimum number of arcs in a $k$-dicritical digraphs when $k \geq 5$. 

\begin{thm}\label{thm_intro:KY}
Let $k \geq 5$ and $G$ a $k$-dicritical digraph. Then \[\abs{A(G)} \ge (k-\frac 1 2 - \frac 1 {k-1}) \abs{V(G)} - k(\frac 1 2 - \frac 1 {k-1})\]
\end{thm}
The way the proof works makes it easy to identify the two arguments that do not allow us to get a better result. 
It is to be noted that our proof works for $k=4$, but in this case a better bound is already known. 

\begin{thm}[Theorem~1 in~\cite{kostochka2020minimum}]
Let $G$ be a $4$-dicritical digraph with $\abs{V(G)} \geq 4$ and $\abs{V(G)} \ne 5$. Then 
\[ \abs{A(G)} \geq \bigceil{\frac{10\abs{V(G)} - 4}{3}}\]
This bound is tight when $n \equiv 1$ (mod $3$) or $n \equiv 2$ (mod $3$)
\end{thm}

Our last result, Theorem~\ref{thm:notLcolour}, has a slightly different flavor than the rest since it deals with list dicolouring. It necessitates a few more technical definitions to be introduced, so we postpone its description to Section~\ref{sec:list} so as not to make this section too heavy. 

\subsubsection*{Organisation of the paper}
Section~\ref{sec:not} is dedicated to notations and Section~\ref{sec:gen} to some basic results that will be needed all along the proofs. 
Section~\ref{sec:dirac} is dedicated to the proofs of Theorems~\ref{thm_intro:directed_dirac_1},~\ref{thm_intro:directed_dirac_2} and~\ref{thm_intro:directed_dirac_k=3}, see respectively subsections~\ref{subsec:dirac},~\ref{subsec:Refined_Dirac} and~\ref{subsec:dirack=3}. Section~\ref{sec:kos} is dedicated to the proof of Theorem~\ref{thm_intro:KY}, Section~\ref{sec:stieb} to the proof of Theorem~\ref{thm_intro:nbcompconnex}. 
Finally, Section~\ref{sec:list} is dedicated to our result on list dicolouring and Section~\ref{sec:furtherworks} to the conclusion.

\section{Notations}
\label{sec:not}

\subsection{Generalities}

We define $\NN = \{0, 1, ...\}$. For $n \in \NN$, we write $[n] = \{1, ..., n\}$ and $\mathfrak S_n$ the set of permutations of $[n]$. Set union will be denoted by $+$ and indexed set union with $\bigcup$. 
Set difference will be denoted by $-$. Excluding a bound of an interval will be denoted by a bracket facing outwards, e.g. $[0, 1[ = \{x \in \RR, 0 \le x < 1\}$.  For $E$ a set and $S \subseteq E$, we denote $\mathbb 1_S$ the indicator function of $S$.

\subsection{Digraphs}

A (\emph{simple}) \emph{digraph} $G$ is a pair $(V(G), A(G))$ where  $V(G)$ is the vertex set and is finite, and $A(G) \subseteq \{(u, v) \in V(G)^2, u \ne v\}$ is the set of arcs of $G$.
The \emph{order} of $G$ is $\abs{V(G)}$. 
We only ever need to consider digraphs up to isomorphism and hence write $G = G'$ whenever $G$ and $G'$ are isomorphic. 

\PA{move to connectivity section(?)} For $X, Y \subset V(G)$, we let $A_G(X, Y) = A(G) \cap (X \times Y)$ and $\bid A_G(X, Y) = A_G(X, Y) + A_G(Y, X)$. 
A \emph{subdigraph} of $G$ is a digraph $G'$ with $V(G') \subseteq V(G)$ and $A(G') \subseteq A(G)$. For $X \subset V(G)$, the subdigraph of $G$ \emph{induced} by $X$ is $G[X] = (X, A(G) \cap X^2)$. For $X \subset V(G)$, we let $G-X = G[V(G)-X]$. 
For $B \subset \{(u, v) \in V(G)^2, u \ne v\}$, we let $G \cup B = (V(G), A(G) + B)$ and $G \setminus B = (V(G), A(G) - B)$. 
For $X$ disjoint from $V(G)$, we let $G+X = (V(G) + X, A(G))$. 

If both $X$ and $V(G)$ are contained in $V(G')$ for some introduced digraph $G'$, we let $G+X = (V(G)+X, A(G)+\bid A_{G'}(V(G), X))$. We denote $\subseteq$ the subdigraph relation, i.e. $G \subseteq H$ whenever $V(G) \subseteq V(H)$ and $A(G) \subseteq A(H)$.

We say that $G$ is \emph{symmetric} when, for any $(u, v) \in A(G)$, $(v, u) \in A(G)$.

\subsection{Arcs,  walks, neighbours, blocks and connectivity}

Let $G$ be a digraph. 

A \emph{digon} of $G$ is a pair of arcs of the form $\{(u, v), (v, u)\}$.
We define
$A^s(G) = \{(u, v) \in A(G)\mid (v, u) \notin A(G)\}$ the set of \emph{simple arcs} of $G$.

A \emph{weak walk} in $G$ is an alternating sequence $P = (x_1, a_1, x_2, \dots, a_{n-1}, x_n)$ of vertices and arcs of $G$, such that, for $i \in [n-1], a_i \in \{(x_i, x_{i+1}), (x_{i+1}, x_i)\}$, we write $V(P) = \{x_1, ..., x_n\}$ and we say that it is a weak walk from $x_1$ to $x_n$. It is a \emph{walk} when, for $i \in [n-1], a_i = (x_i, x_{i+1})$. A \emph{(weak) cycle} is a (weak) walk from a vertex to itself. When $P = (x_1, a_1, x_2, ..., a_{n-1}, x_n)$ is a (weak) walk of $G$, we set $G \setminus P = G \setminus \{a_1, ..., a_{n-1}\}$.

For $X_1, ..., X_n \subseteq V(G)$, the word $X_1...X_n$ denotes $X_1 \times ... \times X_n$. In particular, noticing that giving a walk is the same as giving a sequence of vertices, we denote walks (and cycles) in $G$ as words over $V(G)$, e.g. for $u, v, w \in V(G)$, $uvw$ denotes the walk $(u, (u, v), v, (v, w), w)$. We also write $uv$ for an arc $(u, v)$.

For $X \subseteq V(G)$, we let $N^+(X) = \{u \in V(G)-X, A(X, u) \ne \varnothing\}$ the \emph{out-neighbourhood} of $X$, $N^-(X) = \{u \in V(G)-X, A(u, X) \ne \varnothing\}$ the \emph{in-neighbourhood} of $X$, $N(X) = N^+(X) + N^-(X)$ the \emph{neighbourhood} of $X$, $N^+[X] = N^+(X) + X$ the \emph{closed out-neighbourhood} of $X$, $N^-[X] = N^-(X) + X$ the \emph{closed in-neighbourhood} of $X$ and $N[X] = N(X) + X$ the \emph{closed neighbourhood} of $X$. We also define $N^d(X) = N^+(X) \cap N^-(X)$, $N^s(X) = N(X) \setminus N^d(X)$, $N^{s+}(X) = N^s(X) \cap N^+(X)$ and $N^{s-}(X) = N^s(X) \cap N^-(X)$.

For $x \in V(G)$, we let $d^+(x) = \abs{N^+(x)}$, $d^-(x) = \abs{N^-(x)}$, $d(x) = d^+(x)+d^-(x)$, $d_{min}(x) = \min(d^+(x), d^-(x))$ and $d_{max}(x) = \max(d^+(x), d^-(x))$, respectively the \emph{out-degree}, \emph{in-degree}, \emph{degree}, \emph{min-degree} and \emph{max-degree} of $x$ in $G$.


$G$ is \emph{connected} if, for any $x, y \in V(G)$, there is a weak walk from $x$ to $y$. A \emph{connected component} of $G$ is a maximal set of vertices $X$ such that $G[X]$ is connected. 
We denote $\pi_0(G)$ the set of connected components of $G$.
$G$ is \emph{strongly connected} if and there exists a walk from $u$ to $v$ for every distinct pair of vertices $u,v$. Note that we consider the empty set to be connected, which is not standard, but it simplifies the inductions in the proofs of Section~\ref{sec:stieb}. 

An \emph{arc-cut} of $G$ is a set $A \subseteq A(G)$ of arcs such that $G \setminus A$ is not strongly connected. We say that $G$ is \emph{$k$-arc-connected} when every arc-cut of $G$ has size at least $k$.

A digraph $G$ is \emph{non-separable} if it is connected and $G-v$ is connected for all $v\in V(G)$. Such a vertex is called a \emph{separating vertex} of $G$. 
A \emph{block}  of a digraph $G$ is a subdigraph which is non-separable and is maximal with respect to this property. A block $B$ is a \emph{leaf block} if at most one vertex of $B$ is a separating vertex of $G$, the other blocks are \emph{internal blocks}. 
Observe that if a digraph $G$ is non-separable, then $G$ itself is a leaf block.  Note also that any two distinct blocks of a digraph have at most one vertex in common, and such a common vertex is always a separating vertex of the digraph. 

A \emph{directed Gallai tree} is a digraph whose blocks are either an arc, or a cycle, or a symmetric odd cycle, or a symmetric complete digraph. 
A \emph{directed Gallai forest} is a digraph whose connected components are directed Gallai tree.

\subsection{Basic classes of digraphs and operations on digraphs}

We say that a digraph $G$ is \emph{complete} when $A(G) = \{uv \in V(G), u, v \in V(G)\}$ and we denote by $\bid K_n$ the complete digraph on $n$ vertices.
For $n \geq 2$, $\vec P_n = ([n], \{(i, i+1), i \in [n-1]\})$ is the \emph{path} with $n$ vertices, $n = \vec P_n \cup \{(i+1, i), i \in [n-1]\}$ is the \emph{symmetric path} with $n$ vertices, $\vec C_n = (\ZZ/n\ZZ, \{(i, i+1), i \in \ZZ/n\ZZ)$ is the \emph{cycle} on $n$ vertices and $\bid C_n = \vec C_n \cup \{(i+1, i),  \in \ZZ/n\ZZ\}$ is the \emph{symmetric cycle} on $n$ vertices. A \emph{clique} of a digraph $G$ is a set of vertices inducing a complete digraph.  

For $G$ a digraph and $X_1, ..., X_n$ pairwise disjoint non-empty subsets of $V(G)$, $G/(X_i, i \in [n])$ denotes the digraph obtained from $G$ by merging all vertices in $X_i$, for $i \in [n]$. Formally, let, for $u \in V(G)-\bigcupl_{i \in [n]}X_i, p(u) = u$ and, for $i \in [n]$ and $u \in X_i, p(u) = X_i$. Then $G/(X_i, i \in [n]) = (p(V(G)), \{(p(u), p(v)), (u, v) \in A(G)\})$. $p$ is called the \emph{canonical projection}. When $n = 1$, we write $G/X = G/(X)$. When $X = \{x, y\}$, we denote by $x\star y$ the new vertex resulting from the merging of $x$ and $y$.

If $G$ is a digraph and $G' = (G'_u)_{u \in V(G)}$ is a family of digraphs indexed by the vertices of $G$, the \emph{substitution} $G(G')$ of $G'$ in $G$ is the digraph obtained from $G$ by replacing every vertex by the corresponding digraph. Formally, considering the $V(G'_u), u \in V(G)$ pairwise disjoint, $G(G') = (\bigcupl_{u \in V(G)} V(G'_u), \bigcupl_{u \in V(G)} A(G'_u) + \bigcupl_{(u, v) \in A(G)} V(G'_u)V(G'_v))$. Considering an indexing $u: [n] \to V(G)$ of the vertices of $G$, we write $G(G') = G(G'_{u_1}, ..., G'_{u_n})$.

\subsection{Dicolouring and greedy dicolouring}


Given a digraph $G$ and $X \subseteq V(G)$, we say that $X$ is acylic (in $G$) when $G[X]$ is acylic. 

$\phi: V(G) \to \NN$ is a \emph{dicolouring} of $G$ if, for $n \in \NN, \phi^{-1}(n)$ is acyclic, i.e. has no cycle. A \emph{a $k$-dicolouring} is a dicolouring with colours in $[k]$. 
The \emph{dichromatic number} of $G$ is \[\dic(G) = \min\{n \in \NN, \exists \phi: V(G) \to [n]\text{ dicolouring of }G\}\]

We say that $G$ is \emph{dicritical} when for every proper subdigraph $H$ of $G$, $\dic(H) < \dic(G)$. For $k \in \NN$, we say that $G$ is $k$-dicritical if furthermore $\dic(G) = k$.

Let $G$ be a digraph, $X \subseteq V(G)$, $(u_1, ..., u_n)$ an ordering of the vertices in $G-X$ and $\phi: X \to \NN$ a dicolouring of $G[X]$. \emph{Extending greedily} $\phi$ to $G$ (with respect to the considered ordering) means colouring iteratively $u_1$, ..., $u_n$ so that, for $1 \le i \le n, \phi(u_i) = \min(\NN - \phi(N^-(u_i) \cap (X+u_1+...+u_{i-1})) \cap \phi(N^+(u_i) \cap (X+u_1+...+u_{i-1})))$, i.e. we colour a vertex with the smallest integer that does not appear both in its in-neighbourhood and its out-neighbourhood. When $X = \varnothing$, we say that we colour $G$ greedily.

\subsection{Directional duality}

Any universal statement about digraphs raises a dual statement by exchanging the $+$ and $-$ superscripts, both statements being simultaneously true. 
In particular, a digraph $G$ is $k$-dicritical if and only if the digraph obtained from $G$ by reversing the orientation of each arc is, making directional duality often useful in this context. 
It is out of our scope to give a formal meaning to this so we will use it as an ad hoc principle.

\section{Tools}
\label{sec:gen}

This section is dedicated to basic results that are used all along the proofs. 

\subsection{Basic properties of $k$-dicritical digraphs}

We start with a trivial lower bound on the minimum degree of a vertex in a dicritical digraph. This result will be used so often that we will not refer to it when using it.

\begin{lem}\label{lem:pouet} Let $G$ be a digraph.\begin{enumerate}
    \item\label{lem:d} Let $x \in V(G)$ such that $\dic(G-x) < \dic(G)$. Then, for any $\dic(G-x)$-dicolouring $\phi$ of $G-x$ and $c \in \phi(G-x)$, there is a walk from $N^+(x)$ to $N^-(x)$ in $\phi^{-1}(c)$. In particular, $d_{min}(x) \ge \dic(G)-1$.
    \item\label{lem:shrinkAcyclic} Let $S \subseteq G$ acyclic. Then $\dic(G/S) \ge \dic(G)$
\end{enumerate}
\end{lem}

\begin{proof}\begin{enumerate}
    To prove~\ref{lem:d}, assume towards a contradiction and by directional duality that we have such a $\phi$ and $c$ such that there is no walk from $N^+(x)$ to $N^-(x)$ in $\phi^{-1}(c)$. Then we extend $\phi$ to $G$ by setting $\phi(x) = c$.
    
    We now prove~\ref{lem:shrinkAcyclic}. Let $p : V(G) \to V(G/S)$ be the canonical projection. Let $\phi: V(G/S) \to [\dic(G)-1]$. Let $\psi = \phi \circ p$. Since $\dic(G) \ge k$, $\psi$ is not a dicolouring of $G$. Hence we have a monochromatic cycle $C \subseteq G$. Since $S$ is acyclic, $C \not\subseteq S$. Then, projecting $C$ onto $G/S$ yields a monochromatic cycle in $G/S$. Hence $G/S$ is not $(\dic(G)-1)$-dicolourable.
\end{enumerate}
\end{proof}





\begin{lem}\label{lem:combine}
    Let $G$ be a $k$-dicritical digraph. 
    \begin{enumerate}
        \item Every arc is contained in an induced cycle. 
        \item\label{lem:noOnlyOut} For every $x \in V(G)$, $N^{s+}(x) = \varnothing \Leftrightarrow N^{s-}(x) = \emptyset$.
    \end{enumerate}
\end{lem}

\begin{proof}
    Let $a \in A(G)$. If $a$ is not contained in an induced cycle,  then a $(k-1)$-dicolouring of $G \setminus a$ is a $(k-1)$ dicolouring of $G$.   
    
   Assume one of $N^{s+}(x)$ or $N^{s-}(x)$  is not empty. By directional duality, we may consider $y \in N^{s-}(x)$. By the first point, $yx$ is contained in an induced cycle. The vertex following $x$ in this cycle is in $N^{s+}(x)$.
\end{proof}

\subsection{Basic constructions of $k$-dicritical digraphs}
We now give a simple construction of digraphs with high dichromatic number that will be useful shortly. 

The \emph{Dirac join} of two digraphs $G_1$ and $G_2$ is $\bid K_2(\bid G_1, G_2))$.  

\begin{thm}[\cite{bang2019haj}]\label{lem:addComplete} 
Given two digraph $G_1$ and $G_2$, $\dic(\bid K_2(G_1, G_2)) = \dic(G_1) + \dic(G_2)$, and $\dic(\bid K_2(G_1, G_2))$ is dicritical if and only $G_1$ and $G_2$ are dicritical. 
\end{thm}





We also know how to construct easily dicritical digraphs of any reasonable order. 

\begin{lem}\label{lem:ord} Let $n \ge k \ge 2$. There exists a $k$-dicritical digraph with order $n$.
\end{lem}

\begin{proof} $\bid K_2(\bid K_{k-2}, \vec C_{n+2-k})$ is $k$-dicritical and has $n$ vertices.
\end{proof}

In the symmetric case, it is known that there is no dicritical digraph $G$ with $\dic(G)+1$ vertices. This is not the case for digraphs.

\begin{lem}
Let $k \ge 2$. The only $k$-dicritical digraph with $k+1$ vertices is $\bid K_2(\bid K_{k-2}, \vec C_3)$.
\end{lem}

\begin{proof}
Let $G$ be a $k$-dicritical digraph with $k+1$ vertices. Let $x, y \in V(G)$ such that $xy \notin G$. Let $H = G-x-y$. $G$ is $(\dic(H)+1)$-dicolourable (give the same colour to $x$ and $y$) and hence $\dic(H) \ge k-1$. Since $\abs{V(H)} = k-1$, we obtain $H = \bid K_{k-1}$. Now, since $G \ne \bid K_{k+1}$, we have $x, y \in V(G)$ such that $xy \notin G$. If $yx \notin G$, since $x$ and $y$ have in- and out-degree at least $k-1$, we obtain $G = \bid K_2(G-x-y, \{x, y\})$. Since $G$ is $k$-dicritical, we have a $(k-1)$-dicolouring $\phi$ of $G-x$. Set $\phi(y) = \phi(x)$ to obtain a $(k-1)$-dicolouring of $G$, a contradiction. Hence $y \in N^{s-}(x)$ and then by Lemma~\ref{lem:combine}(\ref{lem:noOnlyOut}) $N^{s+}(x) \ne \varnothing$. Let $z \in N^{s+}(x)$. We proved that $G-x-y = \bid K_{k-1} = G-x-z$. If $yzy \in G$, then $G-x = \bid K_k$ and $G$ is not $k$-dicritical. Hence $G-y-z = \bid K_{k-1}$. In other words, $G = \bid K_2(\bid K_{k-2}, G[\{x, y, z\}])$. By Theorem~\ref{lem:addComplete}, $G[\{x, y, z\}]$ is $2$-dicritical and hence a cycle, which concludes the proof.
\end{proof}

\subsection{Directed Gallai Theorem and directed Gallai forest}

The following theorem is used several times as a tool along the paper, and Section~\ref{sec:list} is dedicated to a slight generalisation of it.

\begin{thm}[Theorem~15 in~\cite{bang2019haj}]\label{thm:directed_Gallai}
If $G$ is a $k$-dicritical digraph, then the subdigraph induced by vertices of degree $2(k-1)$ is a directed Gallai forest. 
\end{thm}

\subsection{Arc-connectivity}

Recall that an \emph{arc-cut} of $G$ is a set $A \subseteq A(G)$ of arcs such that $G \setminus A$ is not strongly connected. We say that $G$ is \emph{$k$-arc-connected} when every arc-cut of $G$ has size at least $k$. It is well known that a digraph $G$ with $|V(G)| \geq 2$ is $k$-arc-connected if and only if for every partition $(V_0, V_1)$ of $V(G)$, we have $|A_G(V_0, V_1)| \geq k$.

There are two technical results about arc-connectivity that will be useful later on: a lower bound on the size of an arc-cut of a $k$-dicritical digraph and a constraint on the dicolouring of digraphs with a small arc-cut.

The next lemma is a generalisation of a classic result on undirected graphs due to Gallai (unpublished) which was also generalised to hypergraphs in~\cite{SST19} (Theorem 12). We could not find any reference for the digraph case.

\begin{lem}\label{lem:lowarcconnection} Let $k \ge 2$, $G$ a $k$-dicritical digraph and $(V_0, V_1)$ a partition of $V(G)$ such that $\abs{A(V_0, V_1)} \le k-1$. Let $V_0^* = N^-(V_1)$ and $V_1^* = N^+(V_0)$. Then there is $i \in \{0, 1\}$ such that, for any $(k-1)$-dicolouring $\phi_i$ of $G[V_i]$, $\abs{\phi_i(V_i^*)} = 1$ and, for any $(k-1)$-dicolouring $\phi_{1-i}$ of $V_{1-i}$, $\abs{\phi_{1-i}(V_{1-i}^*)} = k-1$.
\end{lem}

\begin{proof}
Let, for $i \in \{0, 1\}$, $\phi_i$ be a $(k-1)$-dicolouring of $G[V_i]$. Let $G^*$ be the graph on $\bigsqcupl_{i \in \{0, 1\}} \phi_i(V_i^*)$ such that, for $i \in \{0, 1\}$, $G^*[\phi_i(V_i^*)]$ is complete and, for $c_0 \in \phi_0(V_0^*)$ and $c_1 \in \phi_1(V_1^*)$, $c_0c_1 \in G^*$ if and only if there exists, for $i \in \{0, 1\}$, $x_i \in \phi_i^{-1}(\{c_i\})$ such that $x_0x_1 \in G$. Since $G$ is not $k$-dicolourable, $G^*$ is not $k$-colourable. Since $\overline{G^*}$ is bipartite, it is perfect and hence, by the perfect graph theorem, $G^*$ is perfect. Thus there is $X \subseteq G^*$ such that $G^*[X] = K_k$. Since, for $i \in \{0, 1\}, \abs{\phi_i(V_i^*)} \le k-1$, $X \cap \phi_i(V_i^*) \ne \varnothing$. Since $\abs{E_{G^*}(\phi_0(V_0^*), \phi_1(V_1^*))} \le \abs{A(V_0, V_1)} \le k-1$ and, for $i \in \{0, 1\}$ and $c \in \phi_i(V_i^*)$, $\phi_{1-i}(V_{1-i}^*) \cap N(c) \ne \varnothing$, $\{\abs{\phi_i(V_i^*)}, i \in \{0, 1\}\} = \{1, k-1\}$. This is true for any choice of $\phi_i, i \in \{0, 1\}$, so generalising independently in $\phi_0$ and $\phi_1$ yields the result.
\end{proof}

The above lemma implies the following, which was already proved by Neumann-Lara in~\cite{Neu82} (Theorem 5). 
\begin{corol}\label{lem:arcconnection} Let $k \ge 2$ and $G$ be a $k$-dicritical digraph. Then $G$ is $(k-1)$-arc-connected.
\end{corol}


\section{Dirac-type bounds}\label{sec:dirac}

Let $G$ be a $k$-dicritical digraph.
Every vertex of $G$ has degree at least $2(k-1)$, yielding, by the handshake lemma, $\abs{A(G)} = \frac 1 2 \suml_{u \in V(G)} d(u) \ge (k-1)\abs{V(G)}$. This leads us to define the excess of $u$: $\eps_k(u) = d(u)-2(k-1)$, the excess of $X \subseteq V(G)$: $\eps_k(X) = \suml_{u \in X} \eps_k(u)$ and the excess of $G$: $\eps_k(G) = \eps_k(V(G)) = 2\abs{A(G)} - 2(k-1)\abs{V(G)}$. When it is clear from the context, we write $\eps$ instead of $\eps_k$. 

\subsection{Dirac's Theorem}\label{subsec:dirac}

We now prove Theorem~\ref{thm_intro:directed_dirac_1}, that we restate here for convenience. 
\begin{thm}\label{thm:dirac} Let $n > k \ge 4$ and $G$ an $n$-vertex $k$-dicritical digraph. Then \[\abs{A(G)} \ge (k-1) \abs{V(G)} + k-3.\]
In other words: $\eps(G) \geq 2(k-3)$. 
\end{thm}

\begin{proof} 

Consider a digraph $G$ with $\abs{V(G)} > k$ minimal such that $\eps(G) < 2(k-3)$.

\begin{claim}\label{clm:Kkminus}
$G$ does not contain $\bid K_k$ minus one arc as a subdigraph. 
\end{claim}

\begin{proofclaim}
Assume we have $W \subseteq V(G)$ and $x, y \in W$ such that $G[W]+xy = \bid K_k$. 
Since $G$ is $k$-dicritical and $yx \in A(G)$, $G \setminus yx$ admits a $(k-1)$-dicolouring $\phi$.  
Since $G$ is not $(k-1)$-dicolourable, $\phi(x) = \phi(y)$ and there is a monochromatic walk in $G-yx$ from $x$ to $y$ of colour $\phi(x)$. 
Since $xy \notin A(G)$, this walk has length at least $2$. 
Now, for each $u \in W - \{x, y\}$, define $\psi_u$ from $\phi$ by exchanging the colour of $u$ and the colour of $x$ and $y$; formally: $\psi_u(u) = \phi(x)$, $\psi_u(x) = \psi_u(y) = \phi(u)$ and $\psi_u(v) = \phi(v)$ for every $v \in V(G) - \{x, y, u\}$. 
Since $\psi_u$ is not a dicolouring of $G$, either there is a cycle of colour $\psi_u(x) = \phi(u)$ going through $x$ or $y$ (or both) and we set $\delta_u = 1$, or there is a cycle of colour $\psi_u(u) = \phi(x)$ going through $u$ (which is disjoint from $W - u$) and we set $\delta_u = 0$. 

Observe that if $\delta_u = 0$, then $\eps(u) \geq 2$. Assume $\delta_u=0$ for $c$ vertices. Observe that: $$\eps(W - \{x,y\}) \geq 2c$$
and,
$$\eps(x) + \eps(y) \geq 2\sum_{u \in W - \{x,y\}} \delta_u = 2(k-2-c)$$
Hence, $\eps(G) \geq 2k-4$, a contradiction.
\end{proofclaim}
Note that $\bid K_k \not \subseteq G$, since $G \ne \bid K_k$ and $G$ is $k$-dicritical.

\begin{claim}\label{clm:shrink}
Let $x \ne y \in V(G)$ such that $xy \notin A(G)$ and $G/\{x, y\}$ is not $k$-dicritical. Let $G^* \subseteq G/\{x, y\}$ be $k$-dicritical and $U = V(G)-V(G^*)-x-y$. If $U \ne \varnothing$, then $G^* = \bid K_k$.
\end{claim}

\begin{proofclaim} 
Assume towards a contradiction that $G^* \neq \bid K_k$. 

By minimality of $G$, it suffices to show $\eps(G^*) \le \eps(G)$. 
We have $\varnothing \subsetneq U \subsetneq V(G)$. Hence, by Corollary~\ref{lem:arcconnection}, $G$ is $(k-1)$-arc-connected, so $\abs{A(U, V(G) - U)} \ge k-1$ and $\abs{A(V(G) - U, U)} \ge k-1$.
We have:\[\begin{array}{rcl}
  \eps(G)-\eps(G^*)&=&2(\abs{A(G)}-\abs{A(G^*)}) - (2k-2)(\abs{V(G)}-\abs{V(G^*)})\\
  &=&\suml_{u \in U} d_G(u) + \abs{A(U, V(G) - U)} + \abs{A(V(G) - U, U)}\\
  &&\qquad + 2(\abs{A(V(G) - U)}-\abs{A(G^*)}) - (2k-2)(\abs U + 1)\\
  & \geq & (2k-2)\abs{U} + (2k-2) - (2k-2)(\abs U + 1)\\
  &=&0
\end{array}\]
\end{proofclaim}

\begin{claim} $G$ contains $\bid K_{k-1}$ as a subdigraph. 
\end{claim}

\begin{proofclaim}
We have $x \in V(G)$ such that $d(x) \le 2k-1$ (otherwise $\eps(G) \ge 2\abs{V(G)} \ge 2(k+1)$). By directional duality, we may assume $d^+(x) = k-1$. 
If $N^+(x)$ is a clique, then $G[N^+(x)] = \bid K_{k-1}$ and we are done. 
Otherwise we have $y, z \in N^+(x)$ such that $yz \notin A(G)$. 
Since $d^+_{G/\{y, z\}}(x) < k-1$, $G/\{y, z\}$ is not $k$-dicritical and $x$ is not in any $k$-dicritical subdigraph of $G/\{y, z\}$, so claim~\ref{clm:shrink} yields a copy of $\bid K_{k-1}$ in $G$. 
\end{proofclaim}

Let $W \subseteq V(G)$ such that $G[W] = \bid K_{k-1}$. 
We have $x \in W$ such that $d(x) \le 2k-1$ (otherwise, $\eps(G) \ge \eps(W) \ge 2k-2$). 
Observe that $\abs{N^+(x) - W} = 1$ or $\abs{N^-(x) - W} = 1$. 
Let $y \in N(x) - W$ with $y \in N^+(x)$ whenever $\abs{N^+(x)-W} = 1$ and $y \in N^-(x)$ otherwise. 
We choose such a triplet $(W,\ x,\ y)$ so as to maximise the number of arcs between $x$ and $y$ (i.e. we choose $y \in N^d(x)$ when possible) and, subject to that, maximise the cardinality of $W_y = W \cap N^d(y)$. Let $z \in W - (N^d(y) + x)$ with minimum degree (such a $z$ exists by Claim~\ref{clm:Kkminus}).

By lemma~\ref{lem:pouet}(\ref{lem:shrinkAcyclic}), $\dic(G/\{y, z\}) \ge k$. 
Let $G^*$ be a $k$-dicritical subdigraph of $G/\{y, z\}$ and $U_W = W - (V(G^*) + z)$. 

\begin{claim}\label{clm:WdisjG*}
$U_W = W - z$ and $G^*=\bid K_{k}$. 
\end{claim}

\begin{proofclaim}
We first show $x \in U_W$. 
If $d(x) = 2k-2$ or $y \in N^d(x)$, since $z \in N^d(x)$, $d_{G/\{y, z\}}(x) \leq 2k-3$ and thus $x \notin V(G^*)$.
Otherwise we have $d(x) = 2k-1$ and $y \notin N^d(x)$. Observe that in this case $\abs{N^s(x)} = 3$. 

We may assume by directional duality that $\abs{N^+(x)-W} = 1$ and hence $y \in N^+(x)$. Then $N^{s+}_{G/\{y, z\}}(x) = \varnothing$. If $x \in G^*$, by lemma~\ref{lem:combine}(~\ref{lem:noOnlyOut}), $N^{s-}_{G^*}(x) = \varnothing$ and hence $d_{G^*}(x) \le d_G(x)-3 < 2(k-1)$, a contradiction. So $x \notin G^*$, i.e. $x \in U_W$ and, by Claim~\ref{clm:shrink}, $G^* = \bid K_k$.

Assume towards a contradiction $U_W \subsetneq W - z$. 
Then $1 \le \abs{U_W} \le k-3$.   
Moreover, observe that for every $u \in W-(U_W+z)$, $d_G(u) \geq 2\abs{G^*-u} + 2\abs{U_W} = 2k-2 + 2\abs{U_w}$.
Hence:
\[\begin{array}{rcl}
  \eps(G)&\ge&\eps(W-(U_W+z))\\
  &=&\suml_{u \in W - (U_W + z)} (d_G(u) - (2k-2))\\
  &\ge& \suml_{u \in W - (U_W + z)}(2\abs{U_W})\\
 
  &=&2\abs{W - (U_W + z)}\abs{U_W}\\
  &=&2(k-2-\abs{U_W})\abs{U_W}\\
  &\ge&2(k-3)\qquad (\text{by concavity of }x \mapsto (k-2-x)x),
\end{array}\]

a contradiction.
\end{proofclaim}

Let $R = V(G^*) - y \star z$. By Claim~\ref{clm:WdisjG*}, $G[R] = \bid K_{k-1}$. Let $R_y = R \cap N^d(y)$. The situation is depicted in Figure~\ref{fig:yzWS}.

\begin{figure}
\begin{center}
\begin{tikzpicture}
\draw (0, 0) rectangle (2, 1);
\draw (1, 0) -- (1, 1);

\draw (.5, .5) node {$W_y$};
\draw (1, 1) node[above] {$W-z = \bid K_{k-2}$};
\draw (2, -3) rectangle (3, -1);
\draw (2, -2) -- (3, -2);
\draw (2, -1.5) -- (3, -1.5);

\node (z) at (1, -.75) [vertex] {$z$};
\node (y) at (.25, -2.5) [vertex] {$y$};
\draw (2.5, -2.5) node {$R_y$};
\draw (3, -2) node[right] {$R = \bid K_{k-1}$};
\draw[->, >=latex] (y) -- (.25, 0);
\draw[->, >=latex] (.25, 0) -- (y);
\draw[->, >=latex] (y) -- (2, -2.5);
\draw[->, >=latex] (2, -2.5) -- (y);
\draw[->, >=latex] (z) -- (.5, 0);
\draw[->, >=latex] (.5, 0) -- (z);
\draw[->, >=latex] (z) -- (1.5, 0);
\draw[->, >=latex] (1.5, 0) -- (z);
\end{tikzpicture}
\end{center} 
\caption{This figure describes the situation at the end of claim~\ref{clm:WdisjG*}. $G/\{y,z\}[R \cup y \star z]= {\overset \leftrightarrow K}_{k-1}$}
\label{fig:yzWS}
\end{figure}
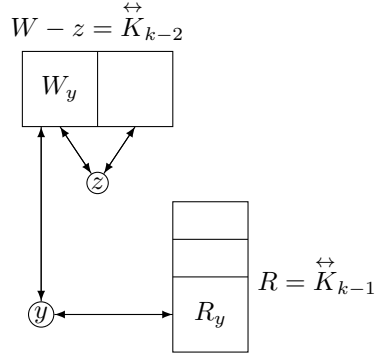

\begin{claim}\label{clm:epsz}
$\eps(z) \ge k-2-\abs{R_y}$. Moreover, if $\abs{R_y} \le k-3$, equality holds only if all the arcs between $z$ and $R-R_y$ have the same orientation.
\end{claim}

\begin{proofclaim}
We have $W-z \subseteq N^d(z)$ and since $G^* = \bid K_k$, we have $R-R_y \subseteq N(z)$. 
Let $s \in R-R_y$ (such an $s$ exists, otherwise $G[R+y] = \bid K_k$). 
We may assume without loss of generality that $s \in N^+(z)$. Since $G$ is $k$-dicritical, we have $\phi: G \setminus zs \to [k-1]$ a dicolouring. 
$\phi$ is not a dicolouring of $G$, so there is a monochromatic walk from $s$ to $z$. Since $W$ is a clique, $z$ is the only vertex in $W$ on the walk. Observe that $s$ is the only element of $R-R_y$ with colour $\phi(s)$, and thus the last but one vertex on the walk is either $s$ or not in $R-R_y$.

Observe moreover that, by Claim~\ref{clm:WdisjG*}, $W-z$ is disjoint from $R$. 
Altogether, we get that $d(z) \geq 2(\abs W - 1) + \abs{R-R_y} + 1 \geq 2(k-2) + (k-1) - \abs{R_y} + 1 = 3k-4-\abs{R_y}$, and thus:

\[\eps(z) = d(z)-(2k-2) \geq k-2-\abs{R_y} \]
Assume now we have $\eps(z) = k-2-\abs{R_y}$, $\abs{R_y} \le k-3$, and for a contradiction $s_+ \in N^+(z) \cap (R-R_y)$ and $s_- \in N^-(z) \cap (R-R_y)$. 
Since $\abs{R-R_y} \ge 2$, we may assume $s_+ \ne s_-$. 
As previously, $G \setminus zs_+$ admits a $(k-1)$ dicolouring, implying that either 
 $s_+ \in N^d(z)$ or $N^-(z)-W-R_y \ne \varnothing$. Similarly, either $s_- \in N^d(z)$ or $N^+(z)-W-R_y \ne \varnothing$, which yields $\eps(z) \ge k-1-\abs{R_y}$.
\end{proofclaim}

Recall that $W_y = N^d(y) \cap W$. 
\begin{claim}\label{clm:epsyz}
$\eps(\{y, z\}) \ge 2\abs{W_y}-2$ and equality holds only if $N^s(y) \setminus R = \emptyset$. 
\end{claim}

\begin{proofclaim}
Since $G^* = \bid K_k$, $\abs{\bid A(\{y, z\}, R)} \ge 2\abs R$. 
Hence:
\[\begin{array}{rcl}
\eps(\{y, z\})&\ge&2\abs R + 2\abs{W_y} + 2(\abs W - 1) - 4(k-1)\\
&=&2\abs{W_y}-2
\end{array}\]
If  $N^s(y) \setminus R \ne \emptyset$, one arc incident to $y$ is not accounted for in the previous minoration.
\end{proofclaim}

\begin{claim}
There is $x' \in R_y$ such that $d(x') \le 2k-1$.
\end{claim}

\begin{proofclaim}
Otherwise, $\eps(R_y) \ge 2\abs{R_y}$. Recall that $z$ has minimum degree among vertices of $W-W_y$. We distinguish two cases:
\begin{itemize}
  \item If $x \in W_y$ (with $w =\abs{W_y}, s = \abs{R_y} \in \intZ{0, k-2}$):
  \[\begin{array}{rcl}
  \eps(G)&\ge&\eps(\{y, z\})+\eps(W-W_y-z)+\eps(R_y)\\
  &\ge&2w - 2 + (k-2-w)(k-2-s)+2s\quad(\text{using Claims~\ref{clm:epsz} and~\ref{clm:epsyz}} )\\
  &=&ws-(k-4)(w+s)+(k-2)^2-2\\
  &=&\frac 1 4 ((w+s)^2 - (w-s)^2) - (k-4)(w+s) + (k-2)^2 - 2
  \end{array}\]
  Let $f(w, s)$ be this last expression. Since, for fixed $w+s$, $f(w, s)$ is decreasing in $\abs{w-s}$ and symmetric in $w$ and $s$, we consider $w', s' \in \intZ{0, k-2}$ such that $w'+s' = w+s$ and $w' \in \{0, k-2\}$ and have:
  \[\begin{array}{rcl}
    \eps(G)&\ge&f(w', s')\\
    &\ge&\min(-(k-4)s' + (k-2)^2-2,\\
    &&\qquad (k-2)s' - (k-4)(k-2+s') + (k-2)^2 - 2)\\
    &\ge&\min((k-2)^2 - (k-2)(k-4) - 2,\\
    &&\qquad (k-2)^2 - (k-2)(k-4) - 2)\\
  &=&2(k-3)\end{array}\]
  \item Otherwise, $x \notin W_y$, that is $y \notin N^d(x)$. 
  Recall that we chose $(W, x, y)$ so as to maximise the number of arcs between $x$ and $y$. 
  Let $u \in W_y \cup R_y$. If $d(u) \le 2k-1$, then either $(W, u, y)$ or $(R, u, y)$ contradicts the choice of $(W, x, y)$. Hence $d(u) \ge 2k$.

  We have $\abs{W_y}, \abs{R_y} \le k-4$ (otherwise $\eps(W_y) \ge 2(k-3)$ (resp. $\eps(R_y) \ge 2(k-3)$)). Then (with $w = \abs{W_y}, s = \abs{R_y} \in \intZ{0, k-4})$):
  \[\begin{array}{rcl}
    \eps(G)&=&\eps(W-W_y-x)+\eps(R_y)\\
    &\ge&(k-2-w)(k-2-s)+2s\qquad(\text{using Claim~\ref{clm:epsz}})\\
    &=&\frac 1 4 ((w+s)^2 - (w-s)^2) - (k-2)(w+s) + (k-2)^2 + 2s\end{array}\]
  This last expression is minimised when $s \le w$ (otherwise exchange $w$ and $s$) and when, for fixed $w+s$, $\abs{w-s}$ is maximised, hence when $s = 0$ or $w = k-4$. Thus we have:\[\begin{array}{rcl}
  \eps(G)&\ge&\min((k-2)(k-2-w),\ 2(k-2-s)+2s)\\
  &\ge&2(k-2)\end{array}\]
\end{itemize}
\end{proofclaim}

Let $x' \in R_y$ with $d(x') \le 2k-1$. 
Since $x' \in N^d(y)$ and we chose $(W, x, y)$ so as to maximise the number of arcs between $x$ and $y$, $x \in W_y$ (otherwise $(R,x',y)$ contradicts the choice of $(W,x,y)$). 

Since we chose $(W, x, y)$ so as to maximise $\abs{W_y}$, we have $\abs{R_y} \le \abs{W_y}$ (otherwise $(R,x',y)$ contradicts the choice of $(W,x,y)$). 
Also recall that $z$ has minimum degree in $W-W_y$. Then:
\[\begin{array}{rcl}
\eps(G)&\ge&\eps(\{y, z\})+\eps(W-W_y-z)\\
&\ge&2\abs{W_y} - 2 + \abs{W-W_y-z}\eps(z)\quad(\text{using Claims~\ref{clm:epsyz}})\\
&\ge&2\abs{W_y} - 2 + (k-2-\abs{W_y})(k-2-\abs{R_y})\quad(\text{using Claims~\ref{clm:epsz}})\\
&\ge&2\abs{W_y}-2+(k-2-\abs{W_y})^2\\
&=&(\abs{W_y}-(k-3))^2+2k-7
\end{array}\]

Since $\eps(G) < 2k-6$, each inequality above is an equality, so $\abs{W_y} = k-3$ and then $\abs{R_y} = k-3$ and equality condition in Claims~\ref{clm:epsz} and~\ref{clm:epsyz} hold. Without loss of generality, we may assume $R-R_y \subseteq N^{s+}(z)$. Since $G^* = \bid K_k$, we have $R-R_y \subseteq N^{s-}(y)$. But since $G$ is $k$-dicritical, by lemma~\ref{lem:combine}(~\ref{lem:noOnlyOut}), $N^{s+}(y) \ne \varnothing$ and hence $N^{s+}(y) \setminus R \ne \varnothing$. This contradicts the equality condition in Claim~\ref{clm:epsyz}.
\end{proof}

\subsection{Refined Dirac's bounds} \label{subsec:Refined_Dirac}

The goal of this section is to prove Theorem~\ref{thm_intro:directed_dirac_2}, that we restate below, together with, as promised in the introduction, the digraph witnessing that the bound is tight.

First, as announced in the introduction, we have to define the set of digraphs $\mathcal D_k$. 


\begin{Def} Let $\m D_3 = \{\bid C_{2n+1}, n \in \NN\}$ and, for $k \ge 4$, let $\m D_k = \{\bid C_5(\bid K_{k-2}, \bid K_1, \bid K_n, \bid K_{k-1-n}, \bid K_1), 1 \le n \le k-2\}$ (see Figure~\ref{fig:D_k}). 
 It is clear that, for every $k \ge 4$ and $G \in \mathcal D_k$, with $a, b \in V(G)$ defined as in Figure~\ref{fig:D_k}, $\eps(\{a, b\}) = 2(k-3)$ and the other vertices have excess $0$, thus $\eps(G) = 2(k-3)$. 
\end{Def}

Observe that all digraphs in $\mathcal D_k$, $k \ge 3$ are symmetric and $k$-dicritical. These digraphs are the same as the tight graphs characterised by Dirac in~\cite{Dirac74} (see Theorem~\ref{thm_nono:dirac2}). 

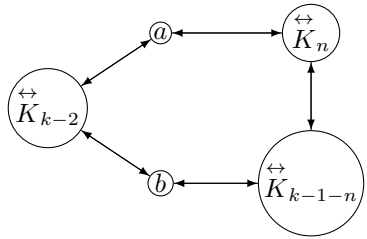
\begin{figure}[!hbtp]
\begin{center}
\begin{tikzpicture}
\node (a) at (0, 0) [vertex] {$\bid K_{k-2}$};
\node (b) at (1.5, 1) [vertex] {$a$};
\node (c) at (1.5, -1) [vertex] {$b$};
\node (d) at (3.5, 1) [vertex] {$\bid K_n$};
\node (e) at (3.5, -1) [vertex] {$\bid K_{k-1-n}$};
\draw[->, >=latex] (a) -- (b);
\draw[->, >=latex] (b) -- (a);
\draw[->, >=latex] (a) -- (c);
\draw[->, >=latex] (c) -- (a);
\draw[->, >=latex] (b) -- (d);
\draw[->, >=latex] (d) -- (b);
\draw[->, >=latex] (c) -- (e);
\draw[->, >=latex] (e) -- (c);
\draw[->, >=latex] (d) -- (e);
\draw[->, >=latex] (e) -- (d);
\end{tikzpicture}
\end{center} 
\caption{Digraphs in $\mathcal D_k$, $k \geq 4$.}
\label{fig:D_k}
\end{figure}

\begin{thm} \label{thm:dirac+} Let $k \ge 4$ and $G$ be a $k$-dicritical digraph such that $G \ne \bid K_k$ and $G \notin \m D_k$. Then:\[\abs{A(G)} \ge (k-1)\abs{V(G)} + (k-2).\]
Equivalently: $\eps(G) \geq 2(k-2)$.\\
Moreover, the bound is tight for $\bid K_2(\bid K_{k-2}, \vec C_3)$. 
\end{thm}

\begin{proof} Assume the theorem is false. Let $k$ be minimal such that the theorem does not hold for $k$. 

Let $G$ be a counterexample of minimal order. Since $\eps(G)$ is even, $\eps(G) \le 2(k-3)$, so by Theorem~\ref{thm:dirac}, $\eps(G) = 2(k-3)$. 
By Theorem~\ref{thm_intro:directed_dirac_1}, we may assume that $G$ is not symmetric.

Let $S = \{u \in V(G)| d(u) = 2(k-1)\} = \{u \in V(G)| \eps(u) = 0\}$. By Theorem~\ref{thm:directed_Gallai}, $S$ induces a directed Gallai-forest.

\begin{claim}\label{clm:dirac+4} $k \ge 5$.
\end{claim}
\begin{proofclaim} 
Assume $k = 4$. We have $\eps(G) = 2$. 
Observe that in this case a block of $S$ is either an arc, or a cycle, or a symmetric odd cycle. 
Moreover, all vertices of $G$ have degree $6$ (and thus are in $S$), except for exactly one vertex of degree $8$ or for exactly two vertices of degree $7$. 
\medskip 

First consider the case where there is a vertex of degree $8$, say $u$. 
Any non-separating vertex of $G[S]$ is in a symmetric odd cycle, since if it were in any other type of block, it would have more than two arcs incident with $u$. This implies that each non-separating vertex of $S$ is linked to $u$ via a digon and that each leaf block of $G[S]$ is a symmetric odd cycle, each of them containing at least $2$ non-separating vertices. 
If $G[S]$ has at most $3$ non-separating vertices, we have $G[S] = \bid K_3$ and then $G = \bid K_4$, a contradiction. 
Hence, there are $4$ of them and thus $N^s(u) = \varnothing$. 
If $G[S]$ has only one block, then this block is $\bid C_4$, a contradiction. Hence $G[S]$ has exactly two leaf blocks, which are $\bid K_3$. Since $G$ is not symmetric, $G[S]$ contains a cycle of length at least $3$, which either leads to another non separating vertex in $S$ or another leaf block, a contradiction.
\medskip 

Now consider the case where there are two vertices of degree $7$, say $u$ and $v$. In particular there are at most $14$ arcs between $\{u,v\}$ and $S$. 

Let $B$ be a block of $G[S]$ and let $x \in V(B)$ be a non-separating vertex of $G[S]$. 
If $B$ is an arc, then $d_G(x) \leq 5$, a contradiction. 
If $B$ is a cycle, then $x$ is linked to both $u$ and $v$ via a digon. 
Assume now that $B$ is a symmetric odd cycle. 
Then $x$ is incident with two arcs incident with $\{u,v\}$. 
Let us prove that there is a digon linking $x$ and $\{u,v\}$. 
Assume towards a contradiction and without loss of generality that $\{ux,xv\} \subseteq A(G)$. 
Let $H$ be obtained from $G$ by removing the arcs $ux$ and $xv$ and adding the arc $uv$. 
Since $x$ is incident with no other simple arc than $ux$ and $xv$, an induced cycle of $G$ that is not a cycle of $H$ contains $uxv$, and thus any dicolouring of $H$ is a dicolouring of $G$. 
Hence $\dic(H) \ge 4$ and $H$ contains a $4$-dicritical subdigraph $H^*$. 
Since each vertex of $H^*$ has degree at least $6$, $x \notin V(H^*)$ and consequently no vertex of $B$ is in $V(H^*)$ (by immediate induction). Since there are at least $4$ arcs between $V(B)$ and $\{u,v\}$, $d_{H^*}(u)+d_{H^*}(v) \le 14-4 = 10$. Hence $u$ or $v$ is not in $H^*$ and $H^* \subsetneq G$, a contradiction.

To summarize, we get that a leaf block of $G[S]$ is either a cycle, and each of its non-separating vertex is linked to both $u$ and $v$ via a digon, or is a symmetric odd cycle, and each of its (at least $2$) non-separating vertex is linked to one of $u$ or $v$ via a digon. Moreover, there is no simple arc between a given non-separating vertex of $G[S]$ and $\{u, v\}$.
In particular, there are at least two digons and no simple arc between the non-separating vertices of a given leaf block and $\{u,v\}$.

For $x \in \{u, v\}$, since $d_G(x)=7$ is odd, $N^s(x) \ne \varnothing$, then by Lemma~\ref{lem:combine}(~\ref{lem:noOnlyOut}), $\abs{N^s(x)} \ge 2$ and then since $d_G(x)$ is odd, $\abs{N^s(x)} \ge 3$, and thus $\abs{N^d(x)} \le 2$. 

This implies that $G[S]$ has at least one internal block. 
And since there are at least two digons between the non-separating vertices of a given leaf block and $\{u,v\}$, we get that $G[S]$ has exactly two leaves blocks $B_1$ and $B_2$, $N^s(u) = N^s(v) = 3$ and $N^d(u) = N^d(v) = 2$, $B_1$ and $B_2$ are either $\bid K_2$ or $\bid K_3$ and  the only digons between $\{u,v\}$ and $S$ are incident with the non-separating vertices of $G[S]$, which are all in $B_1$ or $B_2$.

Assume that $G[S]$ is a symmetric digraph. 
Then $G[S]$ consists in $B_1$ and $B_2$ and a symmetric path $P$ linking $B_1$ and $B_2$. 
Each interior vertex of $P$ is incident to both $u$ and $v$ via simple arcs. Let $H = G\setminus A^s(G) \cup uvu$. Every induced cycle of length at least $3$ in $G$ contains both $u$ and $v$, hence $\dic(H) \ge 4$. Let $H^*$ be a $4$-dicritical subdigraph of $H$. Since $\abs{N^s(u)} = 3$, we have a separating vertex $s$ of $G[S]$ incident to $u$ in $G$. Every vertex of $H^*$ has degree at least $6$, hence $s \notin V(H^*)$. Consequently, since $G[S]$ is connected, $V(H^*) \cap S = \varnothing$, i.e. $V(H^*) \subseteq \{u, v\}$, a contradiction.

So we may assume that one of the internal block is an arc, say $xy$. 
If one of $x$ or $y$, say $x$, is not incident with a $\bid K_3$, then $d_{G[S]}(x) \leq 3$, and since there is no digon between $x$ and $\{u,v\}$, $d_G(x) \leq 5$, a contradiction. 
So both $x$ and $y$ are incident with a $\bid K_3$, and thus $G[S]$ is made of two $\bid K_3$ linked by an arc, namely $xy$. But in this case there are at most $10$ arcs between $S$ and $\{u,v\}$ and thus
 $u$ and $v$ are linked by a digon, a contradiction. 

\end{proofclaim}

Let $$R \in \argmaxl{R \subseteq V(G)\text{ acyclic}} (\eps(R), \abs R)$$ 
Note that $\eps(R) \geq 1$. Note also that, by maximality of $\abs R$, every vertex in $V(G) - R$ has at least one in- and one out-neighbour in $R$. 

\begin{claim} \label{clm:dirac+R} $\eps(R) \ge 2$. \end{claim}

\begin{proofclaim}
Assume $\eps(R) = 1$. By definition of $R$, for every $u \in V(G)$, $\eps(u) \le 1$. $\eps^{-1}(1)$ is a clique, because otherwise we would find an acyclic induced subdigraph of $G$ with excess at least $2$. Furthermore, $\abs{\eps^{-1}(1)} = \eps(G) = 2(k-3)$. 
Then, since $\bid K_k \not\subseteq G$ we have $2(k-3) \leq k-1$ and thus $k = 5$ and $\eps^{-1}(1) = \bid K_4$. Let $u \in \eps^{-1}(1)$. Since $d(u)=2k-1 = 9$ is odd, $N^s(u) \ne \varnothing$ and hence by Lemma~\ref{lem:combine}(~\ref{lem:noOnlyOut}), $\abs{N^s(u)} \ge 2$. Since $d(u)$ is odd, $\abs{N^s(u)} \ge 3$ and thus $\abs{N^s(u)} = 3$. In particular, there is no digon between $\eps^{-1}(1)$ and $S$.

Since $k=5$, every block of $G[S]$ is an arc, a cycle, a symmetric odd cycle or a $\bid K_4$. 
Let $u \in S$ be a non-separating vertex of $G[S]$. 
Since $u$ has degree at most $6$ in $G[S]$, there are at least two simple arcs between $u$ and $\eps^{-1}(1)$. 

Besides, each arc between $u$ and $\eps^{-1}(1)$ is in an induced cycle (because $G$ is dicritical), and thus $u$ is incident with a simple arc in $G[S]$. Then, the block of $G[S]$ containing $u$ is an arc or a cycle and thus there are at least $6$ arcs between $u$ and $\eps^{-1}(1)$, which is impossible.
\end{proofclaim}

\begin{claim}\label{dirac+k-1} $\bid K_{k-1} \subseteq G-R$.
\end{claim}

\begin{proofclaim} Since $R$ is acyclic, $\dic(G-R) \ge k-1$. Let $G^* \subseteq G-R$ be $(k-1)$-dicritical. We may assume $G^* \ne \bid K_{k-1}$.

We have $2(k-3) = \eps_k(G) = \eps_k(V(G)-V(G^*)) + \eps_k(V(G^*))$. 
By claim~\ref{clm:dirac+R}, $\eps_k(V(G)-V(G^*)) \ge 2$. 
By maximality of $\abs R$, each vertex $u \in V(G) - R$ (and thus each vertex in $V(G^*)$) has at least an in- and an out-neighbour in $R$. 
Hence 
\[\eps_k(V(G^*)) \ge \eps_{k-1}(G[V(G^*)]) = \eps_{k-1}(G^*) + 2\abs{A(G[V(G^*)]) - A(G^*)}\]

By Theorem~\ref{thm:dirac}, $\eps_{k-1}(G^*) \ge 2(k-4)$. Altogether, we get: 

\[\begin{array}{rcl}
2(k-3)& = & \eps_k(G)\\
  & = & \eps_k(V(G)-V(G^*)) + \eps_k(V(G^*)) \\
  & \ge & 2 + \eps_{k-1}(G[V(G^*)])\\
  & \ge & 2 + \eps_{k-1}(G^*) + 2\abs{A(G[V(G^*)]) - A(G^*)} \\
  & \ge & 2 + 2(k-4) + 2\abs{A(G[V(G^*)]) - A(G^*)}\\
  & = & 2(k-3) + 2\abs{A(G[V(G^*)]) - A(G^*)}\\
  &\ge&2(k-3)
\end{array}\]

Every inequality is an equality, that is: 

\begin{itemize}
\item $\eps_k(V(G)-V(G^*))=2$, and thus $\eps_k(R) = 2$ by claim~\ref{clm:dirac+R}. 
 \item $\eps_k(V(G^*)) = \eps_{k-1}(G[V(G^*)])$, which implies that for every $x \in V(G^*)$, $\abs{\bid A(x, V(G)-V(G^*))} = \abs{\bid A(x,R)} = 2$, 
 \item $\abs{A(G[V(G^*)]) - A(G^*)} = 0$, that is $G^*$ is an induced subdigraph of $G$, and 
 \item $\eps_{k-1}(G^*) = 2(k-4)$, which implies, by minimality of $k$, that $G^* \in \m D_{k-1}$, 
 
\end{itemize}

Let $a$ and $b$ be the vertices of $G^*$ defined as in Figure~\ref{fig:D_k} (replacing $k$ by $k-1$). Since $\abs{\bid A(a,V(G) - V(G^*))} = \abs{\bid A(b,V(G) - V(G^*))} = 2$ and $a$ and $b$ are non-adjacent, by maximality of $R$ we have $\eps(R) \ge \eps_G(\{a, b\}) = \eps_{G^*, k-1}(\{a, b\}) = 2(k-4)$ and since $\eps(R) = 2$, we obtain $k \le 5$, and thus $k = 5$ by claim~\ref{clm:dirac+4}. 

Hence $G^* \in \mathcal D_4$. Observe that $\mathcal D_4$ contains a single digraph, depicted in Figure~\ref{fig:D_4}. Let $x \in V(G^*)$ as in Figure~\ref{fig:D_4}. 
Since $x$ has (exactly) one in- and one out-neighbour in $R$, $d_G(x) = 10$, and thus $\eps(x) = 2$.

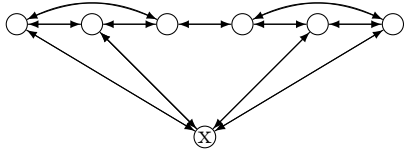
\begin{figure}[!hbtp]
\begin{center}
\begin{tikzpicture}

\node (a) at (0,0) [vertex] {};
\node (b) at (1,0) [vertex] {};
\node (c) at (2,0) [vertex] {};

\node (d) at (3,0) [vertex] {};
\node (e) at (4,0) [vertex] {};
\node (f) at (5,0) [vertex] {};

\node (x) at (2.5, -1.5) [vertex] {x};

\draw[->, >=latex] (a) to (b);
\draw[->, >=latex] (b) to (a);
\draw[->, >=latex] (b) to (c);
\draw[->, >=latex] (c) to (b);
\draw[->, >=latex, bend left = 25] (a) to (c);
\draw[->, >=latex, bend right = 25] (c) to (a);

\draw[->, >=latex] (d) to (e);
\draw[->, >=latex] (e) to (d);
\draw[->, >=latex] (e) to (f);
\draw[->, >=latex] (f) to (e);
\draw[->, >=latex, bend left = 25] (d) to (f);
\draw[->, >=latex, bend right = 25] (f) to (d);

\draw[->, >=latex] (a) to (x);
\draw[->, >=latex] (x) to (a);
\draw[->, >=latex] (b) to (x);
\draw[->, >=latex] (x) to (b);

\draw[->, >=latex] (c) to (d);
\draw[->, >=latex] (d) to (c);

\draw[->, >=latex] (x) to (e);
\draw[->, >=latex] (e) to (x);
\draw[->, >=latex] (x) to (b);
\draw[->, >=latex] (f) to (x);
\draw[->, >=latex] (x) to (f);
\end{tikzpicture}
\end{center} 
\caption{{$G^* = \protect\bid C_5(\protect\bid K_{2}, \protect \bid K_1, \protect\bid K_1, \protect\bid K_{2}, \protect\bid K_{1})$}}
\label{fig:D_4}
\end{figure}
Since $\eps(R) = 2$, by maximality of $\eps(R)$, $x$ is linked by a digon to every vertex with non-zero excess. 
Moreover, since $\bid A(x,R) = 2$, there is only one vertex in $R$ with non-zero excess, say $y$, and thus $\eps(y) = 2$.

 Since $\eps(G) = 2(k-3) = 4 = \eps(\{x,y\})$, every vertex in $V(G) - \{x,y\}$ has excess $0$, i.e. $S = V(G) - \{x,y\}$. In particular, $G^*-x$ is an induced subdigraph of $G[S]$.

Observe that for each vertex $u$ in $S$, $d_G(u)=2(k-1) = 8$ and there are at most $4$ arcs between $u$ and $\{x,y\}$, so $d_{G[S]}(u) \geq 4$. 
This implies that leaf blocks of $G[S]$ are neither $\vec P_2$, nor $\bid K_2$, nor $\vec C_n$. 
Hence, each leaf block of $G[S]$ is either $\bid C_{2n+1}$ for some $n \geq 1$ or $\bid K_4$.

Since $d_G(x) = d_G(y) = 10$ and $x$ and $y$ are linked by a digon, there are $16$ arcs between $S$ and $\{x,y\}$, $8$ between $y$ and $S$, and $8$ between $x$ and $S$ that are already known (see Figure~\ref{fig:D_4}). 

Observe that the number of arcs between the non-separating vertices of a $\bid C_{2n+1}$ leaf block of $G[S]$ and $\{x, y\}$ is $8n$. Moreover, since $G^*-x$ is a subdigraph of $G[S]$, $G[S]$ is not $\bid C_5$. 
Finally, the number of arcs between the non-separating vertices of a $\bid K_4$ leaf block of $G[S]$ and $\{x, y\}$ is $6$. 
Hence, $G[S]$ has at most two leaf blocks and these blocks are either $\bid K_3$ or $\bid K_4$. 

Since $G$ is not symmetric and is dicritical, $G$ contains an induced cycle of length at least $3$. So $G[S]$ is not symmetric.
Hence one of the block of $G[S]$, say $B$, with vertices in $V(G^*)$ is not a leaf block. $B$ contains one of the $\bid K_3$ of $G^*$ and there are $4$ arcs between $V(B)$ and $x$. 

Hence the leaf blocks of $G[S]$ are $\bid K_4$ blocks, there is no arc between a separating vertex of $S$ and $\{x, y\}$ and the non-separating vertices of $G[S]$ are either in $B$ or in a leaf block of $G[S]$. 
If $G[S]$ contains a $\vec P_2$ block $uv$ then, since $G[S]$ has exactly two leaf blocks, $u$ and $v$ are in exactly two blocks of $G[S]$ and hence $d_{G[S]}(u) \le 7$. There is an arc between $u$ and $\{x, y\}$, a contradiction. If $G[S]$ has a $\vec C_n$ block with $n \ge 3$, then this block contains a non-separating vertex of $G[S]$, a contradiction. Hence $G[S]$ is symmetric, a contradiction.
\end{proofclaim}

Let $C = \{x_1, ..., x_{k-1}\} \subseteq V(G)-R$ such that $G[C] = \bid K_{k-1}$ and $d(x_1) \le ... \le d(x_{k-1})$. Let $S' = \{u \in V(G)| d(u) \le 2k-1\} = \{u \in V(G)| \eps(u) \leq 1\}$.
\begin{claim}\label{clm:dirac+epsC} 
For $x_i \in C$, $\eps(x_i) \le \frac{2(k-3)}{k-i+1}$. Thus, $x_1, x_2, x_3 \in S'$.
\end{claim}

\begin{proofclaim} Due to the ordering on the vertices in $C$, we have $\eps(G) \ge \eps(R) + \eps(C) \ge \eps(x_i) + (k-i)\eps(x_i)$. Hence $\eps(x_i) \le \frac{2(k-3)}{k-i+1}$.
\end{proofclaim}

Observe that, since every vertex has in- and out-degree at least $k-1$, each vertex in $C$ has at least one in- and one out-neighbour in $V(G)-C$. 

\begin{claim}\label{clm:dirac+Cy1} Let $y \in V(G)-C$ such that there is $x \in C \cap S'$ with $d^-(x) \le d^+(x)$ and $y \in N^-(x)$ or $d^+(x) \le d^-(x)$ and $y \in N^+(x)$. Then for any $(k-1)$-dicolouring $\phi$ of $G-C$ and $x' \in C$, there is a (possibly empty) monochromatic walk in $G-C$ from $N^+(x')-C$ to $N^-(x')-C$ with colour $\phi(y)$.
\end{claim}

\begin{proofclaim} Let $x \in C \cap S'$ satisfying the hypothesis of the claim. By directional duality, we may assume $d^-(x) \le d^+(x)$ and $y \in N^-(x)$.

We first show the claim in the case $x' \ne x$. Assume towards a contradiction that we have $\phi$ a $(k-1)$-dicolouring of $G-C$ such that there is no monochromatic walk in $G-C$ from $N^+(x')-C$ to $N^-(x')-C$ with colour $\phi(y)$. Set $\phi(x') = \phi(y)$. We want to colour greedily vertices in $C-\{x, x'\}$ from $x_{k-1}$ to $x_1$. To prove this uses only colours in $[k-1]$ we show that, when trying to colour a vertex, it has at most $k-2$ coloured in- or out-neighbours. Let $4 \le i \le k-1$. When colouring $x_i$, $\{x, x_1, ..., x_{i-1}\}-\{x'\}$ is uncoloured and contains at least $i-2$ vertices. Then:\[\begin{array}{rcll}
d_{min}(x_i) - (i-2)&\le&\frac{d(x_i)} 2 - (i-2)&\\
&=&k-1+\frac{\eps(x_i)} 2 - (i-2)&\\
&\le&k-1 + \frac{k-3}{k-i+1} - (i-2)&\text{by claim~\ref{clm:dirac+epsC}}\\
&=&k-1 + \frac 1 {k-i+1} (k-3 - (k-i+1)(i-2))&\\
&\le&k-1 + \frac 1 {k-i+1} (k-3 - 2(k-3))&\text{by convexity and }4 \le i \le k-1\\
&<&k-1&\end{array}\]
Hence we can dicolour greedily $\{x_4, ..., x_{k-1}\}-x$.
Now, for each $u \in \{x_1, x_2, x_3\}-x$, $d(u) \leq 2k-1$ by claim~\ref{clm:dirac+epsC}, and $u$ is connected to $x$ (that is uncoloured) by a digon. Hence we can greedily colour $u$. 
 It remains to colour $x$. We have $x \in S'$ and $d^-(x) \le d^+(x)$. Hence $d^-(x) \le k-1$. Since $y \in N^-(x)$, $x$ has two in-neighbours with the same colour (namely $y$ and $x'$), so we can colour $x$ with a colour from $[k-1]$. We obtain a $(k-1)$-dicolouring of $G$, a contradiction.

If $x' = x$, we apply the claim to $x'' \in \{x_1, x_2, x_3\} - x$ and $y' \in N(x'')-C$ with colour $\phi(y)$ (which exists by the claim applied to $x$, $y$ and $x''$) and $x$ to obtain the result.
\end{proofclaim}

\begin{claim}\label{clm:dirac+sameColour} Let $a \ne b \in V(G)-C$. There exists a $(k-1)$-dicolouring of $G-C$ that gives different colours to $a$ and $b$.
\end{claim}

\begin{proofclaim}
Assume not. Then $\dic(G-C\cup aba) \ge k$. Let $G^* \subseteq G-C \cup aba$ be $k$-dicritical. 

We have: 

\[\begin{array}{rcll}
2(k-3)&=&\eps(G)&\\
&\ge&\eps(V(G^*))& \\
&=&\eps(G^*)-2\abs{A(G^*) - A(G)}\\
&&+ \abs{\bid A(V(G^*), V(G)-V(G^*)} + 2\abs{A(G[V(G^*)])-A(G^*)}& \\
&\ge&2(k-1)-4& \text{by Corollary~\ref{lem:arcconnection} and }A(G^*)-A(G) \subseteq aba\\
&=&2(k-3)
\end{array}\]

Every inequality is an equality, in particular, $\eps(G^*) = 0$, i.e. $G^* = \bid K_k$ by Theorem~\ref{thm:dirac}, and $\abs{\bid A(V(G^*), V(G)-V(G^*))} = 2(k-1)$ by Corollary~\ref{lem:arcconnection}. 
Since, for $x \in \{a, b\}$, $d_{G[V(G^*)]}(x) = 2(k-1) - 2$, we have $a, b \in N(V(G)-V(G^*))$. Since any $(k-1)$-dicolouring of $G[V(G^*)]$ gives the same colour to $a$ and $b$, by Lemma~\ref{lem:lowarcconnection}, any $(k-1)$-dicolouring of $G[V(G^*)]$ gives the same colour to every vertex in $N(V(G)-V(G^*))$, and thus $N(V(G)-V(G^*)) = \{a, b\}$.

Let $H = G-(V(G^*)-a-b)$. Observe that since $G$ is not $(k-1)$-dicolourable, every $(k-1)$-dicolouring of $H$ gives different colours to $a$ and $b$. 
Hence $\dic(H/\{a, b\}) \ge k$, i.e. $H/\{a, b\}$ contains a $k$-dicritical digraph $H^*$. 
If $H^* \ne \bid K_k$, then using Theorem~\ref{thm:dirac}, 

\[\begin{array}{rcll}
\eps(G) & \ge& \eps(V(H^*)-a \star b +a+b)\\
  &\ge &\eps(H^*) + \abs{\bid A(\{a, b\}, V(G^*)-a-b)} - 2(k-1)\\
  &\ge& 2(k-3) + 4(k-2) - 2(k-1)\\
  &\ge& 2(k-2)
\end{array}\]

Hence $H^* = \bid K_k$.

Besides,
\[\begin{array}{rcll}
\eps(G) &\ge& \eps(a, b)\\
 &\ge& 4(k-2) - 4(k-1) + d_{H^*}(a \star b) + \abs{\bid A(a, V(H^*)) \cap \bid A(b, V(H^*)}\\
 &&+ \abs{\bid A(\{a, b\}, V(G)-V(G^*)-V(H^*)}
\end{array}\]

Since $d_{H^*}(a \star b) \ge 2(k-1)$ and $\eps(G) = 2(k-3)$, we obtain $\bid A(a, V(H^*)) \cap \bid A(b, V(H^*)) = \varnothing$ and $\bid A(\{a, b\}, V(G)-V(G^*)-V(H^*)) = \varnothing$. We conclude $G \in \m D_k$, a contradiction.
\end{proofclaim}

Let $y \in N(C \cap S')-C$ satisfying the hypothesis of claim~\ref{clm:dirac+Cy1} (which exists since every vertex in $C$ has an in- and an out-neighbour in $V(G)-C$ and $C \cap S' \ne \varnothing$). If $C \cap S \ne \varnothing$, we choose $y$ to be adjacent to a vertex in $C \cap S$. Up to re-indexing the element of $C$, we may assume that, among the elements of $C \cap S' - S$, the digonal neighbours of $y$ come first, then those that are not adjacent to $y$ and the simple neighbours of $y$ come last. Let $1 \le i' \le k-1$ be minimal such that $y \notin N^d(x_{i'})$ (such an $i'$ exists since $G \ne \bid K_k$).

\begin{claim}\label{clm:dirac+epsy} $\eps(y) \ge \abs{\bid A(y, C)}-\eps(x_{i'})-2$.
\end{claim}

\begin{proofclaim} By claim~\ref{clm:dirac+Cy1}, $G-C \cup (N^-(x_{i'})-C)y(N^+(x_{i'})-C)$ is not $(k-1)$-dicolourable and hence contains a $k$-dicritical digraph $G^*$. Since $G$ is $k$-dicritical, $y \in G^*$. Then:\[\begin{array}{rcl}
d(y)&=&\abs{\bid A(y, C)} + d_{G-C}(y)\\
&\ge&\abs{\bid A(y, C)} + d_{G^*}(y) - \abs{N^-(x_)-C}-\abs{N^+(x_{i'})-C}\\
&\ge&\abs{\bid A(y, C)} + 2(k-1) - (\eps(x_{i'})+2).\end{array}\]
\end{proofclaim}

\begin{claim}\label{clm:dirac+epsi} $\eps(x_{i'}) = 1$, $C \cap S \subseteq N^d(y)$ and $C \cap S' \subseteq N(y)$.
\end{claim}

\begin{proofclaim} By claim~\ref{clm:dirac+epsy} and definition of $i$, $\eps(y) \ge 2(i-2)-\eps(x_i)$. Now, assume $\eps(x_{i'}) \ge 2$. Since $x_{i'} \notin R$ and $\{x_{i'}, y\}$ is acyclic, there is $z \in G-C-y$ with $\eps(z) \ge 1$. We have:\[\begin{array}{rcll} 
\eps(G)&\ge&\eps(y) + \eps(C) + \eps(z)&\\
&\ge&2(i-2) - \eps(x_{i'}) + (k-i)\eps(x_{i'}) + 1&\\
&\ge&2(k-3)+1,\end{array}\]
a contradiction.

Let $x \in C \cap S$ and assume $y \notin N^d(x)$. Let $z \in N(x)-C-y$. By claim~\ref{clm:dirac+sameColour}, we have $\phi$ a $(k-1)$-dicolouring of $G-C$ such that $\phi(y) \ne \phi(z)$, which contradicts claim~\ref{clm:dirac+Cy1}.

As a consequence, by the definition of $i$, $\eps(x_{i'}) = 1$. Now, assume $y$ and $x_{i'}$ are not adjacent. Then by claims~\ref{clm:dirac+sameColour} and~\ref{clm:dirac+Cy1}, $x_{i'}$ has at least two in- and out-neighbours in $V(G)-C$, hence $\eps(x_{i'}) \ge 2$, a contradiction. Thus, by the choice of the ordering on the vertices in $C \cap S'$ and the definition of $i$, $y$ is adjacent to every vertex in $C \cap S'$.
\end{proofclaim}

By claims~\ref{clm:dirac+epsy} and~\ref{clm:dirac+epsi}, we have $\eps(y) \ge 2\abs{S \cap C} + \abs{(S'-S) \cap C} - 3$. Hence:\[\begin{array}{rcl}
\eps(G)&\ge&\eps(y) + \eps(C)\\
&\ge&2\abs{S \cap C} + \abs{(S'-S) \cap C} - 3 + \abs{(S'-S) \cap C} + 2\abs{C-S'}\\
&=&2\abs C - 3\\
&=&2(k-1) - 3\\
&>&2(k-3),\end{array}\]
a contradiction.

\end{proof}

\subsection{Refined Dirac-type bounds for $k=3$}\label{subsec:dirack=3}

The goal of this section is to prove Theorem~\ref{thm_intro:directed_dirac_k=3} that we restate below for convenience. 
We first need to define the set of digraphs $\mathcal D'_3$ mentioned in the introduction.

\begin{Def} 
An \emph{extended wheel} is a digraph made of a vertex $x$ and a triangle $abca$ together with three symmetric paths with lengths of same parity, linking $x$ with $a$, $b$ and $c$ respectively, and such that the three paths have only $x$ in common. One of the paths can be of length $0$, that is $x$ is equal to one of $a$, $b$, $c$, and the two other paths have even length.\\ 
Let $\m D'_3$ be the set of digraphs containing extended wheels and the all digraphs obtained from the digraph pictured in Figure~\ref{fig:D'_3} by replacing any digon by an odd symmetric path.
\end{Def}
It is easy to check that digraphs in $\mathcal D'_3$ are $3$-dicritical, and have excess $2$.

\begin{figure}[!hbtp]
\begin{center}
\begin{tikzpicture}
\node (a) at (0, 0) [vertex] {z};
\node (b) at (-1, -1) [vertex] {x};
\node (c) at (1, -1) [vertex] {y};
\node (d) at (-1, -2) [vertex] {u};
\node (e) at (1, -2) [vertex] {w};
\draw[->, >=latex] (b) -- (a);
\draw[->, >=latex] (c) -- (a);
\draw[->, >=latex] (a) -- (b);
\draw[->, >=latex] (a) -- (c);
\draw[->, >=latex] (c) -- (e);
\draw[->, >=latex] (c) -- (d);
\draw[->, >=latex] (e) -- (d);
\draw[->, >=latex] (e) -- (b);
\draw[->, >=latex] (d) -- (b);
\draw[->, >=latex] (b) -- (c);
\draw[->, >=latex] (d) -- (e); 
\end{tikzpicture}
\end{center} 
\caption{The digraph appearing in the definition of $\mathcal D'_3$}
\label{fig:D'_3}
\end{figure}
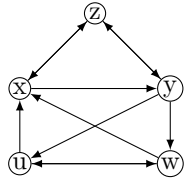

We will also use the following theorem. 
\begin{thm}\cite{ABHR}\label{ABHR}
If $G$ is a $3$-dicritical oriented graph, then $$\abs{A(G)} \geq \frac{7 \abs{V(G)} + 2}{3} $$ 
\end{thm}

\begin{thm}\label{thm:dirac3+} Let $G$ be a $3$-dicritical digraph which is not a symmetric odd cycle. Then $\eps(G) = 2$ if and only if $G \in \mathcal D'_3$, and otherwise $\eps(G)\geq 4$.
\end{thm}

\begin{proof} 
Assume we have a counterexample $G$ of minimal order. 

By Brooks' Theorem, $\eps(G) \ge 1$, and since $\eps(G)$ is even, $\eps(G) = 2$. Thus, either $G$ contains a vertex with excess $2$ or two vertices with excess $1$. 
As usual, let $S = \eps^{-1}(0)$. By Theorem~\ref{thm:directed_Gallai}, $G[S]$ is a directed Gallai forest. Note that, since odd symmetric cycles are $3$-dicritical, the blocks of $G[S]$ are either $\bid K_1$, $\vec P_2$, or cycles. This implies in particular that a non-separating vertex of $G[S]$ is incident with at least $2$ arcs incident with vertices in $V(G) - S$. These facts are constantly used during the proof. 

\begin{claim}\label{clm:dirac+3notoriented}
$G$ has at least one digon. 
\end{claim}

\begin{proofclaim}
Assume towards a contradiction that $G$ has no digon, i.e. $G$ is an oriented graph. 
By Theorem~\ref{ABHR}, $\abs{A(G)} \ge \frac{7 \abs{V(G)} + 2}{3} $. Moreover, since $\eps(G) = 2$, we have $\abs{A(G)} = 2\abs{V(G)} + 2$. We then have $\abs{V(G)} \leq 4$ which is clearly impossible. 
\end{proofclaim}

\begin{claim}\label{clm:noP3}

 Let $P$ be a $\bid P_4$ in $G$. Then the interior vertices of $P$ are not both in $S$.
\end{claim}

\begin{proofclaim}
We proceed by contradiction.
Assume for contradiction that $G$ contains a $\bid P_4$ on vertices $a,b,c,d$ such that $b$ and $c$ are its interior vertices and are in $S$. 
Let $H = G - \{b,c\} + ada$. 

Since $d_G(b) = d_G(c) = 4$, $d_H(a) = d_G(a)$ and $d_G(d) = d_H(d)$, we have $\eps(G) = \eps(H)$. 

Assume that we have a $2$-dicolouring $\phi$ of $H$. Then, by giving colour $\phi(a)$ to $c$, colour $\phi(d)$ to $b$, and colour $\phi(v)$ to every $v \in V(G) - \{b,c\}$, we obtain a $2$-dicolouring of $G$, a contradiction. So $\dic(H) = 3$. 

Let $e$ be an arc of $H$. If $e \notin \{ad,da\}$, then $e \in A(G)$, $G-e$ is $2$-dicolourable, and any $2$-dicolouring of $G-e$ gives distinct colours to $a$ and $d$, so $H-e$ is also $2$-dicolourable. If $e \in \{ad,da\}$, then a $2$-dicolouring of $G-\{b,c\}$ gives distinct colours to $a$ and $d$ (otherwise we can easily extend it to a $2$-dicolouring of $G$), and thus is a $2$-dicolouring of $H-e$. Since $H$ has no isolated vertex, $H$ is $3$-dicritical.

Finally, $H$ is not in $\mathcal D'_3$, for otherwise $G$ is too, contradicting the minimality of $G$. 
\end{proofclaim}

\begin{claim} 
$\forall x \in V(G), \eps(x) \le 1$.
\end{claim}
\begin{proofclaim} 
We assume towards a contradiction that there is $x \in V(G)$ such that $\eps(x) = 2$, i.e. $d(x)= 6$. 
Since $\eps(G) = 2$, we have $V(G)-x = S$. 

For every $s \in S$, $d_{G[S]}(s) \geq 2$ (because $s$ is incident with at most $2$ arcs incident with $x$, and has degree $4$ in $G$). This implies that no connected component of $G[S]$ is a $\bid K_1$ or a $\vec P_2$ and no leaf block of $G[S]$ is a $\vec P_2$. In particular the leaf blocks of $G[S]$ are cycles. 

If a connected component of $G[S]$ is a $\bid K_2$, then it forms a $\bid K_3$ with $x$, a contradiction. If a connected component of $G[S]$ is a $\vec C_3$, then it forms an extended wheel with $x$, a contradiction. If a connected component of $G[S]$ is a cycle of length at least $4$, then $x$ is linked by a digon to each of its vertices, implying that $d(x) \geq 8$, a contradiction. 
So the connected components of $G[S]$ have at least two leaf blocks. 

A leaf block $\vec C_n$, $n \ge 2$ has $n - 1$ non-separating vertices, each of them being connected to $x$ via a digon. 
Thus, $G[S]$ has at most $3$ non-separating vertices, and its leaf blocks are either $\bid K_2$ or $\vec C_3$. 
More precisely, $G[S]$ is connected and its leaf blocks are either three $\bid K_2$, or two $\bid K_2$, or one $\bid K_2$ and one $\vec C_3$. 

Assume first $G$ has two leaf blocks, one $\bid K_2$ and one $\vec C_3$. Since $d_G(x) \le 6$, all the arcs between $x$ and $S$ are incident to a non-separating vertex of one of the leaf blocks, and hence every internal block of $G[S]$ is a $\bid K_2$. 
Since $G$ is not an extended wheel, the path of digons between $x$ and the separating vertex of the $\vec C_3$ leaf block of $G[S]$ has even length, so $G$ is $2$-dicolourable, a contradiction.


Assume now that $G[S]$ has three $\bid K_2$ leaf blocks $\{a_1,b_1\}$, $\{a_2, b_2\}$ and $\{a_3, b_3\}$ such that, for $i=1,2,3$, $a_i$ is a separating vertex of $G[S]$ and $b_i$ is linked by a digon to $x$. 
Since $G$ is $3$-dicritical, for every $2$-dicolouring $\phi$ of $G-x$, we have $\{\phi(b_1), \phi(b_2), \phi(b_3)\} = \{1, 2\}$, and thus $\{\phi(a_1), \phi(a_2), \phi(a_3)\} = \{1, 2\}$, and no proper subdigraph of $G-x$ has this property. 
Hence, the digraph $H$ obtained from $G$ by deleting $b_1, b_2, b_3$ and adding digons between $x$ and $a_i$ for $i=1,2,3$ is $3$-dicritical. 
Moreover, since $\eps(\{b_1, b_2, b_3\}) = 0$ and, for $u \in V(G)-\{b_1, b_2, b_3\}, d_G(u) = d_H(u)$, we have $\eps(H) = \eps(G) = 2$, a contradiction to the minimality of $G$. 

Finally, assume that $G[S]$ has two $\bid K_2$ leaf blocks, say $\{a, b\}$ and $\{c, d\}$ where $b$ and $c$ are separating vertices of $G[S]$. Then $a$ and $d$ are linked to $x$ via a digon. By claim~\ref{clm:noP3}, $b$ is not linked to a vertex of $S-\{a\}$ by a digon, and similarly, $c$ is not linked to a vertex of $S-\{d\}$ by a digon. 
Hence, since $d(x) = 6$, we get that $b$ and $c$ are linked by an arc, as well as $b$ and $x$, and $c$ and $x$, and this gives us a full description of $G$ up to the orientation of the three simple arcs. If $G[\{b,c,x\}] = \vec C_3$, then $G$ is an extended wheel (in which one of the symmetric paths has length $0$), and otherwise $G$ is $2$-dicolourable. A contradiction in both cases. 

\end{proofclaim}

From the previous claim, we get that $G$ has two vertices, say $x$ and $y$, with excess $1$ (i.e. degree $5$), and the other vertices have excess $0$, that is $S = V(G)-\{x,y\}$. 
For $u \in \{x, y\}$, since $d(u) = 5$ is odd, $N^s(u) \ne \varnothing$ and hence, by Lemma~\ref{lem:combine}(~\ref{lem:noOnlyOut}), $\abs{N^s(u)} \ge 2$ and then, since $d(u)$ is odd, $\abs{N^s(u)} \ge 3$. 
In particular, $x$ and $y$ are incident with at most one digon.

\begin{claim}\label{clm:dirac+3nonsepK2} 
Let $u$ be a non-separating vertex of $G[S]$ in a $\bid K_2$ block. Then $N^s(u) = \varnothing$, and thus $u$ is linked to (exactly) one of $x$, $y$ by a digon.
\end{claim}

\begin{proofclaim} 
Assume not. Then $N^s(u) = \{x,y\}$. 
Since the arc between $u$ and $x$ is contained in an induced cycle, we may assume that $yu, ux \in A(G)$, and any induced cycle containing $yu$ or $ux$ contains both $yu$ and $ux$. 
 
This implies that $H = G \setminus \{yu, ux\} \cup yx$ is not $2$-dicolourable (for otherwise $G$ is too) and hence contains a $3$-dicritical digraph $H^*$. 
Observe that $u$ has degree $2$ in $H$, so $u \notin V(H^*)$ and by immediate induction, denoting $S_u$ the connected component of $G[S]$ containing $u$, we have that $S_u \cap V(H^*) = \varnothing$. 
Since $\abs{S_u} \ge 2$, $S_u$ contains at least one other non-separating vertex of $G[S]$, say $w$, and $w$ is incident with two arcs incident with $\{x,y\}$. This implies that $d_{H^*}(x)+d_{H^*}(y) \le 10 +2 - 4 = 8$. Since $x$ and $y$ are in $V(H^*)$ (for otherwise $H^*$ is a subdigraph of $G$), the inequality is an equality, which implies firstly that all vertices of $H^*$ have degree $4$ in $H^*$, and thus $H^*$ is a symmetric odd cycle by Theorem~\ref{thm:brooks}, and secondly that $G[S_u]$ has exactly two non-separating vertices, i.e. $G[S_u]$ is a symmetric path with extremities $u$ and $w$. 

\begin{figure}[!hbtp]
\begin{center}
\begin{tikzpicture}
\node (a) at (0, 0) [vertex] {z};
\node (b) at (-1, -1) [vertex] {x};
\node (c) at (1, -1) [vertex] {y};
\node (d) at (-1, -2) [vertex] {u};
\node (e) at (1, -2) [vertex] {w};
\draw[->, >=latex] (b) -- (a);
\draw[->, >=latex] (c) -- (a);
\draw[->, >=latex] (a) -- (b);
\draw[->, >=latex] (a) -- (c);

\draw[red] (e) -- (c);
\draw[->, >=latex] (c) -- (d);
\draw[->, >=latex] (e) -- (d);
\draw[red] (e) -- (b);
\draw[->, >=latex] (d) -- (b);
\draw[->, >=latex] (b) -- (c);
\draw[->, >=latex] (d) -- (e);

\end{tikzpicture}
\end{center} 
\caption{The digraph at the end of the proof of claim~\ref{clm:dirac+3nonsepK2}. We don't know the orientation of the two red arcs, and there might be a symmetric path of length $2$ linking $u$ and $w$ instead of a digon.}
\label{fig:wheelforclaim}
\end{figure}
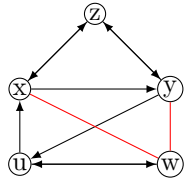

By claim~\ref{clm:noP3}, we have $H^* = \bid K_3$. Since $x$ and $y$ are incident to at most one digon, we have $N^s(w) = \{x, y\}$. Besides, $xy \in A(G)$ and hence, $V(G) = S_u \cup V(H^*)$. By claim~\ref{clm:noP3}, we have $\abs{S_u} \le 3$. If $\abs {S_u} = 3$, colouring $x, y$ and the vertex in $S_u-u-w$ with colour $1$ and the other vertices with colour $2$ yields a $2$-dicolouring of $G$, a contradiction. Hence $S_u = \{u, w\}$. If $ywx \subseteq G$, then $G \in \m D'_3$, a contradiction. Hence, by Lemma~\ref{lem:combine}(~\ref{lem:noOnlyOut}), $xwy \subseteq G$. Every induced cycle containing $xw$ contains $yux$ and hence has a chord (namely $xy$), a contradiction. Hence $G$ is not dicritical, a contradiction. 

\end{proofclaim}

\begin{claim}\label{clm:dirac+3K1} 
$G[S]$ has no $\bid K_1$-block.
\end{claim}
\begin{proofclaim}
Assume $G[S]$ contains a $\bid K_1$ block $\{u\}$. 
Then $u$ is connected to $x$ and $y$ by digons. 
So, there is no digon between $\{x, y\}$ and $S-u$. 
By claim~\ref{clm:dirac+3nonsepK2}, the leaf blocks of $G[S]$ are cycle of length at least $3$. 
Since there are at most $6$ arcs between $\{x, y\}$ and $S-u$, $G[S-u] = \vec C_3$. 
Since $x$ has degree $5$, it cannot be adjacent with $y$ and with the three vertices of the $\vec C_3$. 
So there exists $v \in S-u$ such that $x$ is not adjacent with either $y$ or $v$. 
Hence $\{x, y, v\}$ is acyclic. 
Now, colouring $\{x,y,v\}$ with colour $1$ and the other vertices with colour $2$ yields a $2$-dicolouring of $G$, a contradiction.
\end{proofclaim}

\begin{claim}\label{clm:dirac+3noK2}
$G[S]$ has no $\bid K_2$ leaf block.
\end{claim}

\begin{proofclaim}
Assume towards a contradiction that $G[S]$ contains a $\bid K_2$ leaf block, say $\{u, v\}$, with $u$ non-separating in $G[S]$. By claim~\ref{clm:dirac+3nonsepK2}, we may assume without loss of generality that there is a digon between $x$ and $u$. 

Assume there is $w \in N^d(v)-u$. Then $w \neq x$ because $G$ has no $\bid K_3$. Hence $\{x,u,v,w\}$ contradicts claim~\ref{clm:noP3}.  
So $\abs{N^d(v)} = 1$ and thus $\abs{N^s(v)} = 2$. By claim~\ref{clm:dirac+3nonsepK2} $v$ is separating in $G[S]$. 

Write $N^s(v) = \{a,b\}$, with $b \in S$. 
By Lemma~\ref{lem:combine}(~\ref{lem:noOnlyOut}) and directional duality, we may assume that $bv, va \in A(G)$, and we get that every induced cycle containing $bv$ or $va$ contains both $bv$ and $va$. 
This implies that $H = G\setminus bva \cup ba$ is not $2$-dicolourable, for otherwise so is $G$. So $\dic(H)=3$. Let $H^*$ be a $3$-dicritical subdigraph of $H$. 
Note that every vertex in $V(H^*)$ has degree at least $4$ in $H^*$.  
Hence $v \notin V(H^*)$, which implies $u \notin V(H^*)$. Since $\eps_G(x) = 1$ and $u \in N^d(x)$, if $x \in V(H^*)$, then $\eps_{H^*}(x) < 0$, which is impossible. Hence $x \notin V(H^*)$. As $d_H \le d_G$, we obtain $\eps_G{V(H^*)} \le \eps_G(y) \le 1$. Then, since $G$ is $2$-arc-connected, $\abs{\bid A_G(V(H^*), V(G)-V(H^*)} \ge 4$. As we added exactly one arc when constructing $H$, we obtain $\eps(H^*) \le \eps_G(V(H^*))+2-4 < 0$, a contradiction.

\end{proofclaim}

\begin{claim}\label{clm:dirac+3P1} 
$G[S]$ has no $\vec P_2$ leaf block.
\end{claim}
\begin{proofclaim} 
Assume there is a $\vec P_2$ leaf block in $G[S]$, say $\{u, v\}$ with $u$ non-separating in $G[S]$. 
We may assume without loss of generality that there is a digon between $u$ and $x$, and a simple arc between $u$ and $y$. 
Moreover, since the arc between $u$ and $v$ is in an induced cycle, we may assume that $vuy \subseteq G$ and we get that all induced cycle going through $vu$ goes through $uy$. 
This implies that 
 $H = G \setminus vu \cup yu$ is not $2$-dicolourable and hence contains a $3$-dicritical digraph $H^*$. 
Let $S_v$ be the connected component of $v$ in $G[S]$. Every vertex in $H^*$ has degree at least $4$ in $H^*$, and $v$ has degree $3$ in $H$, so $v \notin V(H^*)$ and an immediate induction shows that $V(S_v) \cap V(G^*) = \{u\}$. 
In $G$, $S_v-u$ contains a non-separating vertex of $G[S]$, which is incident with (at least) two arcs incident with $\{x,y\}$. So $d_{H^*}(x)+d_{H^*}(y) \le 10 - 2 + 1 = 9$. Hence $\eps(H^*) \le 1$. Since $\eps(H^*)$ is even, $\eps(H^*) = 0$ and thus $H^*$ is a symmetric odd cycle. Now, since $u \in V(H^*)$ (for otherwise $H^*$ is a subgraph of $G$), we get that $x \in V(H^*)$. So $x$ is incident with two digons in $H^*$ and thus in $G$, a contradiction.  

\end{proofclaim}

\begin{claim}\label{clm:dirac+32connect} 
$G[S]$ has exactly two leaf blocks, which are cycles of length at least $3$. Moreover, there are at least $8$ arcs between the non-separating vertices of $G[S]$ and $\{x,y\}$.
\end{claim}
\begin{proofclaim} By claims~\ref{clm:dirac+3K1},~\ref{clm:dirac+3noK2} and~\ref{clm:dirac+3P1}, every leaf block of $G[S]$ is a cycle of length at least $3$. For each such block $B$, we have $\abs{\bid A(B, \{x, y\})} \ge 4$ and since $d(x)+d(y) = 10$, there are at most two of them.

Assume towards a contradiction that $G[S]$ has only one leaf block. Then $G[S]$ has only one block which is a cycle of length at least $3$.

Assume first that there is no arc between $x$ and $y$. Then $G[S] = \vec C_5$. Since $\abs{N^d(x)} \le 1$ and $\abs{N^d(y)} \le 1$, we have $s \in S-N^d(x)-N^d(y)$. Then, colouring $x, y, s$ with colour $1$ and all other vertices with colour $2$ yields a $2$-dicolouring of $G$, a contradiction.

 Assume now that there is a simple arc between $x$ and $y$, say $xy \in A(G)$. Then $G[S] = \vec C_4$, say $G[S] = s_1s_2s_3s_4s_1$. By claim~\ref{clm:dirac+3notoriented}, $G$ contains a digon. Assume without loss of generality that there is a digon between $x$ and $s_1$. So $x$ is non-adjacent with one of the vertices $s_i$ of $S$, $i \ne 1$. If there is no digon between $y$ and $s_i$, then colouring $x,y,s_i$ with colour $1$, and the other vertices with colour $2$ yields a $2$-dicolouring of $G$, a contradiction. So there is a digon between $y$ and $s_i$. Hence $y$ is non-adjacent to some vertex in $S$. Let $s_j \in S$ with $j \ne 1$ and $j \ne i$ such that $y$ is non-adjacent to a vertex in $S-s_i-s_j$ (which exists since $\abs S = 4$). Then, colouring $x$, $s_i$ and $s_j$ with colour $1$ and the other vertices with colour $2$ yields a $2$-dicolouring of $G$, a contradiction.

Finally, assume that there is a digon between $x$ and $y$. Then $G[S]$ is a cycle of length $3$ and there is no digon between $S$ and $\{x,y\}$ (because $\abs{N^d(x)} \le 1$ and $\abs{N^d(y)} \le 1$)). By Lemma~\ref{lem:combine}(~\ref{lem:noOnlyOut}), $x$ has both an in- and an out-neighbour in $S$. By directional duality, we may assume $\abs{N^{s+}(x)} = 2$. Now, colouring $N^+[x]$ with colour $1$ and the rest of the vertices with colour $2$ yields a $2$-dicolouring of $G$, a contradiction.

Hence, $G[S]$ has exactly two leaf blocks, which are cycles of length at least $3$. Each of these leaf blocs have at least two non-separating vertices, and each of these vertices are incident with two arcs incident with $\{x,y\}$. So the second part of the statement holds. 
\end{proofclaim}

\begin{claim}\label{clm:dirac+3digonxyS} There is no digon between $S$ and $\{x, y\}$.
\end{claim}
\begin{proofclaim} Assume there is such a digon. 
Without loss of generality, assume there exists $u \in N^d(x) \cap S$. 

If $u$ is separating in $G[S]$, then $u$ is in two $\vec P_2$ blocks and hence, by claim~\ref{clm:dirac+3P1}, its neighbours in $S$ are separating in $G[S]$ and are each incident with at least one arc incident with $\{x, y\}$. Hence there are at most $6$ arcs between the non-separating vertices of $G[S]$ and $\{x, y\}$, which is impossible by claim~\ref{clm:dirac+32connect}.

Hence $u$ is non-separating in $G[S]$. 
Let $B$ be the block of $G[S]$ containing $u$. 
By claim~\ref{clm:dirac+3noK2}, $B$ is not a $\bid K_2$ block, so $B$ is a cycle of length at least $3$. 
Let $u^- \in N^{s-}(u)$ and $u^+ \in N^{s+}(u)$. 
Since the only induced cycle going through $uu^+$ or $u^-u$ is $B$, 
$H = G \setminus u u^+ \cup u u^-$ is not $2$-dicolourable and hence contains a $3$-dicritical digraph $H^*$. Since every vertex in $H^*$ has degree at least $4$, an immediate induction on the walk from $u^+$ to $u^-$ in $B$ shows that $V(H^*) \cap V(B) \subseteq \{u, u^-\}$. In particular, $d_{H^*}(u^-) \le 4$. 
Since $u^+ \notin V(H^*)$, there is a non-separating vertex of $G[S]$ that is not in $H^*$. Hence, if $x, y \in V(H^*)$, then $d_{H^*}(x)+d_{H^*}(y) \le 9$. In any case, $\eps(G^*) \le 1$. So $\eps(G^*) = 0$ and thus $H^*$ is a symmetric odd cycle. We have $u, u^- \in V(H^*)$, for otherwise $H^*$ is a subdigraph of $G$. Since $d_{H^*}(u) = 4$, $x \in V(H^*)$ and thus $x$ is incident with two digons, a contradiction.
\end{proofclaim}

\begin{claim}\label{clm:dirac+3digonS} There is no digon in $S$.
\end{claim}

\begin{proofclaim} 
Let $P$ be a maximal symmetric path in $G[S]$ and let $u$ and $v$ its extremities. Assume towards a contradiction that $P$ has length at least $1$, i.e. $u \ne v$. If both $u$ and $v$ are in $\vec P_2$ blocks, then the extremities of these two $\vec P_2$ are separating vertices by claim~\ref{clm:dirac+32connect}, and thus each of (the four of) them is adjacent to $x$ or $y$. Hence there are at most $6$ arcs between the non-separating vertices of $G[S]$ and $\{x, y\}$, contradicting claim~\ref{clm:dirac+32connect}.

Hence we may assume that $u$ is not in a $\vec P_2$ block. By maximality of $P$, it is not in a second $\bid K_2$ block. Hence it is in a cycle of length at least $3$. 
Let $u^- \in N^{s-}(u)$ and $u^+ \in N^{s+}(u)$. 
Since an induced cycle containing $uu^+$ or $u^-u$ contains both $uu^+$ and $u^-u$, $H = G \setminus u^-uu^+ \cup u^- u^+$ is not $2$-dicolourable. 
So $H$ contains a $3$-dicritical digraph $H^*$. 
Since every vertex in $H^*$ has degree at least $4$, an immediate induction on the component of $G[S] \setminus u^-uu^+$ containing $u$ finds a non-separating vertex of $G[S]$ which is not in $H^*$. 
So $\eps(H^*) \le 1$, and thus $\eps(H^*) = 0$ and thus $H^*$ is a symmetric odd cycle. If $V(H^*) \subset S$, then $G[S]$ contains a symmetric cycle minus one arc, which is impossible. Hence $V(H^*) \cap \{x, y\} \ne \varnothing$, which contradicts claim~\ref{clm:dirac+3digonxyS}.

\end{proofclaim}

By claim~\ref{clm:dirac+3notoriented}, $G$ contains a digon. By claims~\ref{clm:dirac+3digonxyS} and~\ref{clm:dirac+3digonS}, there is a digon between $x$ and $y$. 
Hence there are $6$ arcs between $S$ and $\{x, y\}$, a contradiction to claim~\ref{clm:dirac+32connect}.

\end{proof}

\section{Kostochka-Yancey-type bound}
\label{sec:kos}

The main goal of this paper is to obtain the best bounds on the minimum number of arcs in dicritical digraphs with fixed order and dichromatic number. One way of doing so, is to search for such bounds as linear functions of the order and search for the best slope. We give here a nice characterisation of this quantity.

The directed Haj\'os join describes a way to build $k$-critical digraphs from any two $k$-dicritical digraphs, with the following properties:

\begin{lem}[Theorem~2 in~\cite{bang2019haj}]\label{lem:hajos}
Let $k \ge 2$ and let $G_1$ and $G_2$ be $k$-dicritical digraphs. Then there exists a $k$-dicritical digraph $G$ with 
$\abs{A(G)} = \abs{A(G_1)} + \abs{A(G_2)} -1$ and 
$\abs{V(G)} = \abs{V(G_1)} + \abs{V(G_2)} -1$
\end{lem}

\begin{lem}\label{lem:fknasympt} Let $k \ge 2$ and, for $n \ge k$, $d_k(n)$ be the minimum number of arcs in a $k$-dicritical digraph of order $n$. Then: \[\frac 1 n d_k(n) \quad\tend{n \to +\infty}\quad \infl_{G\ k\text{-dicritical}} \frac{\abs{A(G)}-1}{\abs{V(G)}-1} \in [k-1, k-\frac 2 {k-1}]\]
\end{lem}

\begin{proof} Given two integers $n$ and $m$, we write $n \% m$ for the rest of the euclidean division of $n$ by $m$.  First, by Lemma~\ref{lem:ord}, $d_k$ is well-defined. Then, by Lemma~\ref{lem:hajos}, we have, for $a, b \ge k$:\[d_k(a+b-1) \le d_k(a)+d_k(b)-1\]
and hence, for $a \ge b$, \[\begin{array}{rcl}
d_k(a)&=&d_k(b + (b-1) \bigfloor{\frac{a-b}{b-1}} + (a-b)\%(b-1))\\
&=&d_k(b + (a-b)\%(b-1)) + \suml_{i = 0}^{\bigfloor{\frac{a-b}{b-1}}-1} (d_k(b+(a-b)\%(b-1) + (i+1)(b-1)) \\
&&\phantom{d_k(b + (a-b)\%(b-1)) + \suml_{i = 0}^{\bigfloor{\frac{a-b}{b-1}}-1}}- d_k(b+(a-b)\%(b-1)+i(b-1)))\\
&\le&d_k(b+(a-b)\%(b-1)) + \bigfloor{\frac{a-b}{b-1}} (d_k(b)-1)\end{array}\]
i.e. \[\frac 1 a d_k(a) \le \frac{d_k(b)-1}{b-1} + O(\frac 1 a)\]
This yields $\limsupl_{n \to +\infty} \frac 1 n d_k(n) \le \infl_{G\ k\text{-dicritical}} \frac{\abs{A(G)}-1}{\abs{V(G)} - 1}$. But it is immediate that:\[\infl_{G\ k\text{-dicritical}} \frac{\abs{A(G)}-1}{\abs{V(G)} - 1} \le \liminfl_{n \to +\infty} \frac 1 n d_k(n)\]
and the result follows (the upper bound comes from Theorem~\ref{thm:KY}).)
\end{proof}

\subsection{Minimum number of arcs in a $k$-dicritical digraph}

The goal of this section is to prove Theorem~\ref{thm_intro:KY} that we restate below for convenience, see Theorem~\ref{thm:master}. 
\medskip

 Let $G$ be a digraph. Two distinct vertices $u, v \in V(G)$ are \emph{twins} in $G$ when $N^+[u] = N^+[v]$ and $N^-[u] = N^-[v]$. In particular a pair of twins are linked by a digon.

\begin{Def} Let $G$ be a digraph, $R \subseteq G$, and $\phi : R \to [k-1]$ be a dicolouring of $G[R]$. For $i \in [k-1]$, let $X_i = \phi^{-1}(i)$. We define $Y(G, R, \phi)$ as the digraph obtained from $G$ after contracting each $X_i$ into a single vertex $x_i$, and adding a digon between $x_i$ and $x_j$ for every $i \neq j$. 

\end{Def}

\begin{lem}\label{lem:_partShrinkDic} Let $G$ be a digraph, $R \subseteq V(G)$, and $\phi$ be a $(k-1)$-dicolouring of $G[R]$. If $\dic(G) \ge k$, then $\dic(Y(G, R, \phi)) \ge k$.
\end{lem}

\begin{proof} 
By lemma~\ref{lem:pouet}(\ref{lem:shrinkAcyclic}) and because adding arcs does not decrease the dichromatic number.
\end{proof}

We will also need the following technical lemma. 
\begin{lem}[Lemma~17 in~\cite{KoYa14}]\label{lem:KY}
Let $k \ge 3$, $R_* = \{u_1, \dots, u_s\}$ be a set, and $\omega: R_* \to \NN^*$ such that $\omega(u_1) + \dots + \omega(u_s) \ge k-1$. 
Then for each $1 \leq i \leq (k-1)/2$, there exists a graph $H$ with $V(H) = R_*$ and $\abs{E(H)} = i$ such that for every independent set $M$ in $H$ with $\abs M \geq 2$,
\[
\sum_{u \in R_*-M} \omega(u) \geq i
\]
\end{lem}

Our aim is to show the following theorem:

\begin{thm}\label{thm:master} 
For every $k$-dicritical digraph $G$, 
\[\abs{A(G)} \ge (k-\frac 1 2 - \frac 1 {k-1}) \abs{V(G)} - k(\frac 1 2 - \frac 1 {k-1})\]
\end{thm}

\begin{proof} 
Let $\eps \in ]0, \frac 1 2 - \frac 1 {k-1}[$. 
Define the potential of a digraph $G$ as follows: \[\rho(G) = (k-1+\eps)\abs{V(G)} - \abs{A(G)}\]
and for $R \subseteq V(G)$, the potential of $R$ in $G$ is $\rho_G(R) = \rho(G[R])$. 
\medskip

Let us first discuss the potential of cliques.

\begin{claim} For $i \ge 1$, $\rho(\bid K_i) = i(k-i+\eps)$. In particular:\begin{itemize}
 \item $\rho(\bid K_1) = k-1+\eps$
 \item $\rho(\bid K_{k-1}) = (k-1)(1+\eps)$
 \item $\rho(\bid K_k) = k\eps$.
\end{itemize}
Besides, $\rho(\bid K_k) < \rho(\bid K_1) < \rho(\bid K_{k-1}) < \minl_{2 \le i \le k-2} \rho(\bid K_i)$ (the last inequality can be seen easily using the concavity of $i \mapsto \rho(\bid K_i)$).
\end{claim}

Note that if $G$ is a digraph and $H$ is a spanning proper subdigraph of $G$, then $\rho(G) \leq \rho(H)$. In particular $\rho(\bid K_{\abs{V(G)}}) \leq \rho(G)$. These two easy facts are often used in the proof. 
\medskip 

We are going to show that, for any $k$-dicritical digraph $G$, $\rho(G) \le \rho(\bid K_k) = k\eps$. 
This indeed implies the theorem because we get that $\abs{A(G)} \ge 2(k-1+\eps)\abs{V(G)} - k\eps$. This being true for any $\eps \in ]0, \frac 1 2 - \frac 1 {k-1}[$, it also holds for $\eps = \frac 1 2 - \frac 1 {k-1}$, which gives $\abs{A(G)} \ge (k-\frac 1 2 - \frac 1 {k-1}) \abs{V(G)} - k(\frac 1 2 - \frac 1 {k-1})$ as wanted.

\medskip

We order the digraphs lexicographically on 
\[G \mapsto (\abs{V(G)}, \abs{A(G)}, \abs{A^s(G)}, -\abs{\{(u, v) \in V(G)^2, d(u) = d(v) = 2(k-1) \land u\text{ and }v\text{ are twins}\}})\] 
(denoting $\preceq$ the ordering) and consider a $\preceq$-minimal counter-example $G$. So $\rho(G) > \rho(\bid K_k)$ and we minimise the number of vertices, then the number of arcs, then the number of simple arcs, and finally we maximise the number of twins of degree $2(k-1)$. 

Let $S = \{u \in V(G), d(u) = 2(k-1)\}$.

We start the proof by a lower bound on the potential of a subset of $V(G)$.

\begin{claim}\label{clm:_noSmallPot} 
Let $R \subsetneq V(G)$. If $\abs R \ge 2$, then $\rho_G(R) > \rho(\bid K_1) = k-1-\eps$.
\end{claim}

\begin{proofclaim} 
Let $R \in \argminl{\substack{W \subsetneq V(G)\\\abs W \ge 2}}\rho_G(W)$. Towards a contradiction, we assume $\rho_G(R) \le \rho(\bid K_1)$.

Since $\rho(\bid K_{\abs R}) \le \rho_G(R) \le \rho(\bid K_1) < \minl_{2 \le i \le k-1} \rho(\bid K_i)$, we have $\abs R \ge k$.
Since $R \subsetneq V(G)$ and $G$ is $k$-dicritical, we have a dicolouring $\phi: G[R] \to [k-1]$. Let $Y = Y(G, R, \phi)$ and $X = V(Y)-V(G)$. Since $\dic(G) = k$, by lemma~\ref{lem:_partShrinkDic} we have $\dic(Y) \ge k$ and hence $Y$ contains a $k$-dicritical subdigraph $Y^*$.

Since $\abs R \ge k$, $\abs{V(Y^*)} \le \abs{V(Y)} = \abs{V(G)}-\abs R+(k-1) < \abs{V(G)}$, so $Y^* \prec G$ and hence $\rho(Y^*) \le \rho(\bid K_k)$. 

Since $G$ is $k$-dicritical, $Y^* \not\subseteq G$ and hence $X \cap V(Y^*) \ne \varnothing$. So, $\rho(\bid K_1) \leq \rho(\bid K_{\abs{V(Y^*) \cap X}}) \leq \rho_{Y^*}(V(Y^*) \cap X)$.

We have:\[\begin{array}{rcl}
\rho_G(V(Y^*) - X + R)&=&\rho_G(V(Y^*) - X) + \rho_G(R) - \abs{\bid A_G(V(Y^*)- X, R)}\\
&\le&\rho(Y^*) - \rho_{Y^*}(V(Y^*) \cap X) + \rho_G(R)\\
&&\qquad+ \abs{\bid A_{Y^*}(V(Y^*) - X, V(Y^*) \cap X)} - \abs{\bid A_G(V(Y^*) - X, R)}\\
&\le&\rho_G(R) + \rho(\bid K_k) - \rho(\bid K_1)\\
&<&\rho_G(R)\end{array}\]
Since $2 \le \abs R \le \abs{V(Y^*) - X + R}$, by minimality of $R$, $V(Y^*) - X + R = V(G)$ and thus:\[
\rho(G) \le \rho(\bid K_k) + \rho_G(R) - \rho(\bid K_1) \le \rho(\bid K_k),\]
a contradiction. 

\end{proofclaim}

We are now ready to obtain a much stronger lower bound.

\begin{claim}\label{clm:_bigPot} 
Let $R \subsetneq G$ such that $\abs R \ge 2$. If $\rho_G(R) \le \rho(\bid K_{k-1}) = (k-1)(1+\eps)$, then $G[R] = \bid K_{k-1}$.
\end{claim}

\begin{proofclaim} Let $R \in \argminl{\substack{W \subsetneq V(G)\\\abs W \ge 2 \land G[W] \ne \bid K_{k-1}}}\rho_G(W)$. Towards a contradiction, we assume $\rho_G(R) \le \rho(\bid K_{k-1})$. 
Since $\rho_G(R) \le \rho(\bid K_{k-1}) < \minl_{2 \le i \le k-2} \rho(\bid K_i)$ and $G[R] \ne \bid K_{k-1}$, we have $\abs R \ge k$.

Let $i = \bigceil{\rho_G(R)-\rho(\bid K_k)}-1$, so that $\rho(\bid K_k)+i < \rho_G(R) \le \rho(\bid K_k)+i+1$. 
By claim~\ref{clm:_noSmallPot}, we have $k-1 + \eps = \rho(\bid K_1) < \rho_G(R) \le \rho(\bid K_k)+i+1$ and hence since $\eps < \frac 1 2 - \frac 1 {k-1}$, we have $i > k-1+\eps - k\eps - 1 > \frac{k-1} 2$. 
In particular $i \ge 2$. 

Besides, we have $\rho(\bid K_k)+i < \rho_G(R) \le \rho(\bid K_{k-1})$ and hence $i < (k-1)(1+\eps)-k\eps = k-1-\eps$ which gives $i \le k-2$. 

Since by Corollary~\ref{lem:arcconnection} $\abs{\bid A(N(V(G)-R), G-R)} \ge 2(k-1)$, Lemma~\ref{lem:KY} with $\omega: x \in N(V(G)-R) \mapsto \abs{\bid A(x, G-R)}$ implies the existence of a set of digons $A$ with end vertices in $N(V(G)-R)$ of size $\bigfloor{\frac i 2}$ such that for every $I \subseteq N(V(G)-R)$ with $\abs I \ge 2$ and independent in the digraph $(N(V(G)-R), A)$, we have $\abs{\bid A(N(V(G)-R) - I, V(G)-R)} \ge \bigfloor{\frac i 2}$. 

We show that $G[R] \cup A$ is $(k-1)$-dicolourable. If it is not the case, we have $G^* \subseteq G[R] \cup A$ $k$-dicritical. Then, $\rho(G^*) \ge \rho_G(G^*)-2\bigfloor{\frac i 2} \ge \rho_G(R)-i > \rho(\bid K_k)$, which contradicts the minimality of $G$. 

Let $\phi: R \to [k-1]$ be a dicolouring of $G[R] \cup A$. Let $Y = Y(G, R, \phi)$ and $X = V(Y)-V(G)$. Since $\dic(G) = k$, by lemma~\ref{lem:_partShrinkDic} we have $\dic(Y) \ge k$ and hence $Y$ contains a $k$-dicritical subdigraph $Y^*$. Since $\abs R \ge k$, we have $\abs{V(Y^*)} < \abs{V(G)}$, that is $Y^* \prec G$. By minimality of $G$, $\rho(Y^*) \le \rho(\bid K_k)$. Since $G$ is $k$-dicritical, $Y^* \not\subseteq G$ and hence $X \cap V(Y^*) \ne \varnothing$. We have:\[\begin{array}{rcl}
 \rho_G(Y^*-X+R)&=&\rho_G(Y^*-X) + \rho_G(R) - \abs{\bid A(Y^*-X, R)}\\
 &=&\rho_Y(Y^*-X) + \rho_G(R) - \abs{\bid A(Y^*-X, R)}\\
 &=&\rho_Y(Y^*)-\rho_Y(Y^* \cap X) + \rho_G(R) + \abs{\bid A(Y^*-X, Y^* \cap X)} - \abs{\bid A(Y^*-X, R)}\\
 &\le&\rho(Y^*)-\rho_Y(Y^* \cap X) + \rho_G(R) + \abs{\bid A(Y^*-X, Y^* \cap X)} - \abs{\bid A(Y^*-X, R)}\\
\end{array}\]
If $\abs{Y^* \cap X} \ge 2$, we obtain: $\rho_G(Y^*-X+R) \le \rho(\bid K_k)-\rho(\bid K_{k-1})+\rho_G(R) \le \rho(\bid K_k) < \rho(\bid K_1)$, a contradiction. Hence $\abs{Y^* \cap X} = 1$. Then: $\rho_G(Y^*-X+R) \le \rho(\bid K_k) - \rho(\bid K_1) + \rho(\bid K_k) + i+1 - \bigfloor{\frac i 2}$. By claim~\ref{clm:_noSmallPot}, we have $\rho_G(Y^*-X+R) > \rho(\bid K_k)$. We obtain $k-1+\eps - k\eps < i+1-\bigfloor{\frac i 2} \le i+1-\frac{i-1} 2 = \frac{i+3} 2 \le \frac{k+1} 2$ and hence $\eps \ge \frac 1 2 - \frac 1 {k-1}$, a contradiction.
\end{proofclaim}

We are now going to show some strong structural properties of $G$.

\begin{claim}\label{clm:_arcAdd} 
Let $R \subsetneq V(G)$ and $A$ be a set of at most $k-2$ arcs with end vertices in $R$. Then $G[R] \cup A$ is $(k-1)$-dicolourable.
\end{claim}

\begin{proofclaim}
Otherwise, let $G^* \subseteq G[R] \cup A$ be $k$-dicritical. We have $\abs R \ge \abs{V(G^*)} \ge k$. In particular $G[R] \neq \bid K_{k-1}$, so $\rho_G(R) > \rho(\bid K_{k-1})$. Hence $\rho(G^*) \ge \rho_G(V(G^*)) - (k-2) > \rho(\bid K_{k-1}) - (k-2) = (k-1)(1+\eps) - (k-2) = k\eps +1-\eps \ge \rho(\bid K_k)$, and since $G^* \prec G$, we get a contradiction with the minimality of $G$.
\end{proofclaim}

\begin{claim}\label{clm:_neighSmall} 
Let $u \in G$ with $d(u) \le 2k-1$. Then $N^d(u) = \varnothing$ or $N^s(u) = \varnothing$. 

\end{claim}

\begin{proofclaim}
We proceed by contradiction. By directional duality, we may assume $\abs{N^+(u)} \ge \abs{N^-(u)}$. Let $N^{s+}(u) = \{x^+_i, 1 \le i \le t\}$ and $N^{s-}(u) = \{x^-_i, 1 \le i \le s\}$ (with $s \le t$). If $s = 0$, by lemma~\ref{lem:combine}(~\ref{lem:noOnlyOut}), $t = 0$ and $N^s(u) = \varnothing$. If $s = k-1$, then $N^d(u) = \varnothing$. 
So $1 \leq s \leq k-2$. 

Let $H = G \setminus \{u x^+_i, 1 \le i \le s\} \cup uN^{s-}(u)$. 
Any dicolouring of $H$ is a dicolouring of $G$, so $\dic(H) \ge k$. Let $H^* \subseteq H$ be $k$-dicritical. Since $s \leq k-2$, we have $\abs{uN^{s-}(u)} \leq k-2$ and thus, by claim~\ref{clm:_arcAdd}, $V(H^*) = V(G)$. Note that $H^* \prec G$ (because we moved the arcs so as to create digons). We have $\rho(H^*) \ge \rho(G) > \rho(\bid K_k)$, a contradiction to the minimality of $G$.
\end{proofclaim}

\begin{claim}\label{clm:_smallInOut} Let $x, y \in V(G)$ such that $xy \in A(G)$, $yx \notin A(G)$, $d^+(x) = k-1$ and $d(y) \le 2k-1$. Then $d^-(y) = k$. In particular, any pair of vertices in $S$ are either non adjacent, or linked by a digon.
\end{claim}

\begin{proofclaim}
Assume towards a contradiction that $d^-(y) = k-1$. By claim~\ref{clm:_neighSmall}, we have $z \in N^{s-}(y)-x$. Let $H = G-x-y \cup zN^{s+}(y)$.

We have $\chi(H) \ge k$. Otherwise, consider $\phi: H \to [k-1]$ a dicolouring. Since $d_{G-y}^+(x) < k-1$, we can extend $\phi$ into a $(k-1)$-dicolouring of $G-y$. Since $\phi$ cannot be extended into a $(k-1)$-dicolouring of $G$, we have $\phi(N^-(y)) = [k-1]$. Since $\abs{N^-(y)} = k-1$, $\phi$ is injective on $N^-(y)$. Set $\phi(y) = \phi(z)$. Let $C$ be a monochromatic cycle. We have $z' \in N^+(y)$ such that $zyz' \subseteq C$. Then $C \setminus zyz' \cup zz'$ is a monochromatic cycle in $G-y$, a contradiction. 

Let $H^* \subseteq H$ $k$-dicritical. Since $H^*$ is not a subdigraph of $G$, $z \in V(H^*)$ and at least one of the added arc is in $A(H^*)$. Then:\[\begin{array}{rcl}
\rho_G(V(H^*)+y)&=&\rho(H^*) + \rho(\bid K_1) - (\abs{A_G(V(H^*)+y)} - \abs{A(H^*)})\\
&\le&\rho(\bid K_k) + \rho(\bid K_1) - 1\\
&=&\rho(\bid K_{k-1})+2\eps - 1\\
&<&\rho(\bid K_{k-1}).\end{array}\]
Since $x \notin V(H^*)$, $V(H^*)+y \ne V(G)$ and we obtain a contradiction to Claim~\ref{clm:_bigPot}.
\end{proofclaim}

\begin{claim}\label{clm:_sameNeigh} Let $X = \bid K_{k-1} \subseteq G$ and $x, y \in X \cap S$. Then $x$ and $y$ are twins.
\end{claim}

\begin{proofclaim}
By claim~\ref{clm:_neighSmall}, $N(x) = N^d(x)$ and $N(y) = N^d(y)$. Let $u_x \in N(x)-X$ and $u_y \in N(y)-X$. Assume towards a contradiction that $u_x \ne u_y$. Let $H = G-x-y \cup u_xu_yu_x$. By claim~\ref{clm:_arcAdd}, we have $\phi: H \to [k-1]$ a dicolouring. We have $\phi(u_x) \ne \phi(u_y)$. If $\phi(u_y) \in \phi(X-x-y)$, we take $\phi(x) \in [k-1]-\phi(X-x-y)$, otherwise we set $\phi(x) = \phi(u_y)$. In both cases $y$ has two neighbours with the same colour and hence we can extend $\phi$ greedily to $G$, a contradiction. 
\end{proofclaim}

A set of vertices $C$ of $G$ is a \emph{cluster} if $C \subseteq S$, $C$ is a clique, each pair of vertices in $C$ are twins, and $C$ is maximal with these properties.

\begin{claim}\label{clm:_smallCluster} Let $C$ be a cluster of $G$. Then $\abs C \le k-3$. 
\end{claim}

\begin{proofclaim}
By claim~\ref{clm:_neighSmall}, a cluster of size at least $k-2$ would be at most $2$ arcs away from being a $\bid K_k$, contradicting claim~\ref{clm:_arcAdd}. 
\end{proofclaim}

\begin{claim}\label{clm:_digonCluster} Let $x, y \in S$ such that there is a digon between $x$ and $y$, $x$ (resp. $y$) is in a cluster of size $s$ (resp. $t$), $x$ and $y$ are not twins and $t \le s$. Then $x$ is in a $\bid K_{k-1}$ and $t = 1$.
\end{claim}

\begin{proofclaim} 
By claim~\ref{clm:_neighSmall}, $N^s(x) = \varnothing$. 
Let $G' = G-y+x'$ so that $N[x'] = N^d[x'] = N[x]$ (i.e. $x$ and $x'$ are twins and linked by a digon). We have $\abs{V(G')} = \abs{V(G)}$. Since $y \in N_G^d(x)$ and $x' \in N^d_{G'}(x')$, $d_G(y) = d_G(x) = d_{G'}(x) = d_{G'}(x')$ and hence $\abs{A(G')} = \abs{A(G)}$. Furthermore, $N_{G'}^s(x') = \varnothing$, so $\abs{A^s(G')} \le \abs{A^s(G')}$. Removing $y$ reduces the number of twins by $2(s-1)$. Subsequently, adding $x'$ increases the number of twins by $2t$. Since $t \le s$, we conclude $G' \prec G$.

Assume we have $\phi': G' \to [k-1]$ a dicolouring. Set, for $u \in V(G)-\{x,y\}$, $\phi(u) = \phi'(u)$, then $\phi(y) \in [k-1]-\phi'(N(y)-x)$ and finally $\phi(x) \in \{\phi'(x), \phi'(x')\}-\{\phi(y)\}$. It is easy to check that $\phi$ is a $(k-1)$-dicolouring of $G$, a contradiction. 

Hence $\dic(G') \ge k$. Let $G^* \subseteq G'$ be $k$-dicritical. We have $G^* \prec G$, so $\rho(G^*) \le \rho(\bid K_k)$. Besides, $G$ is $k$-dicritical and hence $x' \in V(G^*)$. We have $\rho_G(V(G^*)-x') \le \rho(\bid K_k) - \rho(\bid K_1) + 2(k-1) = \rho(\bid K_{k-1}$). Since $y \notin V(G^*)-x'$, by claim~\ref{clm:_bigPot}, $G^*-x' = \bid K_{k-1}$. Finally, since $x' \in V(G^*)$, $G^*$ is $k$-dicritical and $d(x') = 2(k-1)$, we have $x \in N[x'] = V(G^*)$. Hence $x$ is in a $(k-1)$-clique in $G$.

Now, $N[x]-y = \bid K_{k-1}$. If the cluster of $y$ contains a vertex $x' \in N[x]-y$, then $x'$ and $y$ are twins and thus $N[x]$ is $\bid K_k$, a contradiction. So the cluster of $y$ is disjoint from $N[x]-y$, but any vertex in the cluster of $y$ is a neighbour of $x$, so the cluster of $y$ is reduced to $y$, i.e. $t=1$. 
\end{proofclaim}

\begin{claim}\label{clm:_clusterNeighBig} Let $C$ be a cluster with $\abs C \ge 2$.\begin{enumerate}
 \item\label{clm:_cNB1} If $\bid K_{k-1} \not\subseteq G[N[C]]$, then $\forall u \in N(C), d(u) \ge 2(k-1+\abs C)$. 
 \item\label{clm:_cNB2} If there is $X \subseteq N[C]$ such that $G[X] = \bid K_{k-1}$, then $\forall u \in X-C, d(u) \ge 2(k-1+\abs C)$.
\end{enumerate}\end{claim}

\begin{proofclaim} Assume towards a contradiction that we have $u \in N(C)$ such that $d(u) < 2(k-1+\abs C)$ and, if there is $X \subseteq N[C]$ such that $G[X] = \bid K_{k-1}$, $u \in X-C$. 

Assume $u \in S$. For $c \in C \cap N(u)$, since $\abs C \ge 2$, by claim~\ref{clm:_neighSmall}, $u \in N^d(c)$. By claim~\ref{clm:_digonCluster}, since $\abs C \ne 1$, 
$G[C] \subseteq \bid K_{k-1}$. 
Then $\bid K_{k-1} \subseteq G[N[C]]$ and hence by definition of $u$, there is $X \subseteq N[C]$ such that $G[X] = \bid K_{k-1}$ and $u \in X-C$. 
Since $d(u) = 2(k-1)$, there is $c \in C \cap X$. By claim~\ref{clm:_sameNeigh}, $u$ and $c$ are twins, i.e. $u \in C$, a contradiction. 
So $d(u) \ge 2k-1$.

Let $c \in C$ and $G' = G-u+c'$ where $c'$ is a new vertex such that $N^+[c'] = N^+[c]$ and $N^-[c'] = N^-[c]$, i.e. $c$ and $c'$ are twins. 
Assume we have $\phi'$ a $(k-1)$-dicolouring of $G'$. 
Set, for $x \in G-C-u$, $\phi(x) = \phi'(x)$. 
Then take $\phi(u) \in [k-1]-(\phi(N^+(u)-C)\cap \phi(N^-(u)-C))$ (which is not empty since $d(u) < 2(k-1+\abs C)$) and then colour $C$ with colours in $\phi'(C+c') - \phi(u)$. This is a $(k-1)$-dicolouring of $G$, a contradiction. 
Hence $\dic(G') \ge k$ and $G'$ contains a $k$-dicritical digraph $G^*$. Since $d(u) \ge 2k-1 > 2(k-1) = d_{G'}(c')$, $\abs{A(G')} < \abs{A(G)}$ and hence $G' \prec G$. 
Hence $\rho(G^*) \le \rho(\bid K_k)$. Since $G^* \not \subseteq G$, we have $c' \in V(G^*)$. Since $d(c') = 2(k-1)$, we obtain $C \subseteq V(G^*)$.
We have: $\rho_G(G^*-c') \le \rho(G^*) - \rho(\bid K_1) + 2(k-1) \le \rho(\bid K_{k-1})$. Hence by claim~\ref{clm:_bigPot}, $G^*-c' = \bid K_{k-1}$. We have $N[C]-u = \bid K_{k-1}$. Hence, by the choice of $u$, there is $X \subseteq N[C]$ such that $G[X] = \bid K_{k-1}$ and $u \in X$. Let $v \in N[C]-X$. 
Then $N[C] \cup uvu = \bid K_k$, A contradiction to claim~\ref{clm:_arcAdd}.
\end{proofclaim}

We are now going to obtain a contradiction using the discharging method. Let $\alpha = \frac \eps {k-2}$. Each $u \in V(G)$ starts with charge $d(u)$. We apply the following rules (observe that any charge sent through an arc is at least $\alpha$):
\begin{itemize}
 \item Every vertex with degree at least $2k$ keeps $2(k-1+\eps)$ to himself and distributes the rest equally along its arcs: it sends charge $\frac{d(u)-2(k-1+\eps)}{d(u)}$ through each of its arcs. Note that this expression increases with $d(u)$ and hence is at least $\frac{1-\eps} k \ge \alpha$.
 \item Every vertex with degree $2k-1$ and $k$ out-neighbours (resp. $k$ in-neighbours) sends charge $\alpha$ to its out-neighbours (resp. in-neighbours).
 \item For every $u \in S$ such that $u$ is in a cluster of size at least $2$ which is in a $(k-1)$-clique $X$, $u$ sends $2\alpha$ to its unique neighbour that is not in $X$. 
\end{itemize}
The uniqueness of the neighbour of $u$ in the last bullet is due to claim~\ref{clm:_neighSmall}. Indeed, since $u$ is in a $(k-1)$-clique, $N^s(u) = \varnothing$ and hence $\abs{N(u)} = k-1$.

Let, for $u \in V(G)$, $w(u)$ be its resulting charge. We are going to prove that for every $u \in V(G)$, $w(u) \geq 2(k-1+\eps)$.

\begin{itemize}
 \item Let $u \in V(G)$ such that $d(u) \ge 2k$. Then by construction, $w(u) = 2(k-1+\eps)$.
 \item Let $u \in V(G)$ such that $d(u) = 2k-1$ and $d^-(u) = k-1$. 
 By claim~\ref{clm:_neighSmall}, $N^d(u) = \varnothing$. By claim~\ref{clm:_smallInOut}, for every $x \in N^-(u)$, $d^+(x) \geq k$. So $u$ receives charge (at least $\alpha)$ through $k-1$ arcs and sends $\alpha$ through $k$ arcs. Hence $w(u) \ge d(u)-\alpha \ge 2(k-1+\eps)$.
 \item Let $u \in S$ such that $u$ is in a cluster of size $1$. So $u$ does not send any charge. Claims~\ref{clm:_neighSmall} distinguishes two cases.\\
 Assume first $N^d(u) = \varnothing$. Then by claim~\ref{clm:_smallInOut}, for every $y \in N^-(u)$, either $d(y) \geq 2k$, or $d^+(y) \geq k$. In both cases $y$ sends at least $\alpha$ to $u$. The same holds for the out-neighbours of $u$. So $w(u) = d(u) + 2(k-1)\alpha \ge 2(k-1+\eps)$. 
 
 Assume now $N^s(u) = \varnothing$. 
 By claim~\ref{clm:_neighSmall}, no neighbour of $u$ has degree $2k-1$.
 If $u$ is in a $(k-1)$-clique of $G$, by claim~\ref{clm:_sameNeigh}, every neighbour of $u$ in this clique has degree at least $2k$, and hence sends charge to $u$. Hence $w(u) \ge d(u) + 2(k-2)\alpha \ge 2(k-1+\eps)$. Assume this is not the case. Let $v \in N(u)$. If $d(v) \ge 2k$, then $v$ sends $2\alpha$ to $u$. Otherwise, $v \in S$. Since $u$ is not in a $(k-1)$-clique of $G$, by claim~\ref{clm:_digonCluster}, $v$ is in a cluster of size at least $2$ and in a $(k-1)$-clique. Hence, by the third rule, $v$ sends $2\alpha$ to $u$. Thus, $w(u) = d(u) + 2(k-1)\alpha \ge 2(k-1+\eps)$.
 
 \item Let $u \in S$ such that $u$ is in a cluster $C$ of size $c \ge 2$. Note that by claim~\ref{clm:_neighSmall}, $N^s(u) = \varnothing$. \\
 If $\bid K_{k-1} \not\subseteq G[N[C]]$, then $u$ does not send any charge and, by claim~\ref{clm:_clusterNeighBig}~\ref{clm:_cNB1}, 

 it has $k-1+c$ neighbours of degree at least $2(k-1+c) \ge 2k$ and hence send charge towards $u$ by rule 1: 
 \[w(u) \ge d(u) + 2(k-c)\frac{2(k-1+c)-2(k-1+\eps)}{2(k-1+c)}\] 
 Otherwise, let $X \subset N[C]$ such that $G[X] = \bid K_{k-1}$ and $u \in X$. By claim~\ref{clm:_clusterNeighBig}~\ref{clm:_cNB2}, all vertices in $X-C$ have degree at least $2(k-1+c) \ge 2k$ and hence send charge towards $u$. Finally, $u$ sends charge to at most one vertex (its neighbour that is not in $X$): \[w(u) \ge d(u) + 2(k-1-c)\frac{2(k-1+c)-2(k-1+\eps)}{2(k-1+c)} - 2\alpha\] 
 In both cases, $w(u) \ge 2(k-1) + 2(c-\eps)\frac{k-1-c}{k-1+c} - 2\frac{\eps}{k-2}$.
 We have:
 \[\begin{array}{rcl}
 w(u) \ge 2(k-1+\eps)&\Leftrightarrow&(k-2)(c-\eps)(k-1-c) -\eps(k-1+c) -(k-2)(k-1+c)\eps \ge 0\\
 &\Leftrightarrow&(2(k-1)(k-2) + k-1+c) \eps \le (k-2)c(k-1-c)\end{array}\]
 The first expression is concave in $c$, so by claim~\ref{clm:_smallCluster}, we only have to check it for $c \in \{2, k-3\}$. 
 Since $\eps < \frac 1 2 - \frac 1 {k-1}$, we only need to check $(2(k-1)(k-2) + k-1+c) (\frac 1 2 - \frac 1 {k-1}) \le (k-2)c(k-1-c)$. 
 For $c = 2$, we obtain: $(k-3)(2k^2-7k+7) \ge 0$, which is true since the degree $2$ polynomial has discriminant $-7$ and hence is always positive. 
 For $c = k-3$, we obtain $(k-3)(2k^2-8k+7) \ge 0$, which is true since the largest root of the polynomial of degree 
 $2$ is $2+\frac 1 {\sqrt 2}$. 
\end{itemize}

Hence $\abs{A(G)} = \frac 1 2 \suml_{u \in G} d(u) = \frac 1 2 \suml_{u \in G} w(u) \ge (k-1+\eps)\abs{V(G)}$, i.e. $\rho(G) \le 0$, a contradiction.
\end{proof}

\section{Generalisation of a result of Stiebitz}\label{sec:stieb}
\label{sec:sti}

The goal of this section is to prove Theorem~\ref{thm_intro:nbcompconnex}.

Recall that $\pi_0(G)$ denotes the set of connected components of $G$. 
We are actually going to prove the following stronger statement:

\begin{thm}\label{thm:nbcompconnex_contr} Let $G$ be a connected digraph, $k \ge 3$ and $X \subseteq V(G)$ such that:\begin{itemize}
  \item $\forall u \in X, d(u) \le 2(k-1)$.
  \item $\forall S \in \pi_0(G[X]), \dic(G-S) \le k-1$
  \item $\abs{\pi_0(G-X)} > \abs{\pi_0(G[X])}$
\end{itemize}
Then $\dic(G) \le k-1$.
\end{thm}

We will need the following definition. 

\begin{Def} For $G$ a digraph, $X \subseteq V(G)$ and $P$ a partition of $\pi_0(G-X)$, we define the following (undirected) bipartite graph: 
\[B(G, X, P) = (\pi_0(G[X]) + P, \{ST | S \in \pi_0(G[X]), T \in P, \bid A(S, \bigcupl_{C \in T} C) \ne \varnothing\}).\]

Let $B$ be a bipartite graph with partite sets $U$ and $V$. A $2$-forest of $B$ with respect to $U$ is a spanning forest of $B$ in which every vertex in $U$ has degree $2$.
\end{Def}

The following remark describes a method to extend the dicolouring of a partially dicoloured digraph that will be used a lot duting the proof. 

\begin{Rq}\label{rq:greedyExtension} 
Let $G$ be a digraph, $H \subseteq G$ connected, $x \in V(H)$ and $\phi$ a $(k-1)$-dicolouring of $G-H$. Assume that, for every $u \in V(H), d_G(u) \le 2(k-1)$. 
Then, given the reverse ordering of a BFS of the underlying graph of $H$ starting in $x$, $\phi$ can be be greedily extended to $G-x$ (because, when colouring $u \in V(H)$, $u$ is incident with at most $2k-3$ arcs incident with an already coloured vertex).\\
Moreover, if $\phi(N^+(x)) \neq [1, k-1]$ or $\phi(N^-(x)) \neq [k-1]$, then $\phi$ can be extended to $G$.

\end{Rq}

The next Lemma is a strong version of Theorem~\ref{thm:nbcompconnex_contr} in the case where $\abs{\pi_0(G[X])}=1$. 

\begin{lem}\label{lem:case1} Let $G$ be a connected digraph and $X \subseteq V(G)$ such that:\begin{itemize}
  \item $\forall u \in X, d(u) \le 2(k-1)$
  \item $G[X]$ is connected
  \item $G-X$ is disconnected
\end{itemize}
Then, for any $(k-1)$-dicolouring $\phi$ of $G-X$, there is a $(k-1)$-dicolouring $\psi$ of $G$ so that $\forall C \in \pi_0(G-X), \exists \sigma \in \mathfrak S_{k-1}, \phi_{|C} = \sigma \circ \psi_{|C}$.\\

\end{lem}

\begin{proof}
We proceed by induction on $\abs X$. The result is trivial when $X = \varnothing$. 
Let $\phi$ be a $(k-1)$-dicolouring of $G-X$.
Let $x \in X$ such that $G[X-x]$ is connected (any leaf on a spanning tree of $G[X]$ suits). 

Assume first that $G-(X-x)$ is disconnected. Let $S \in \pi_0(G-(X-x))$
such that $x \in S$. Since $G$ is connected, $X-x \ne \varnothing$ and hence, since $G[X]$ is connected, $d_{G[S]}(x) < 2(k-1)$. So we can extend $\phi$ to $G-(X-x)$ and then apply induction on $X-x$.

Assume now that $G-(X-x)$ is connected. So, for all $S \in \pi_0(G-X), S \cap N(x) \ne \varnothing$. Let $S_0 \ne S_1 \in \pi_0(G-X)$. 
We can permute colours in $S_0$ and in $S_1$ so that $x$ has neighbours in both $S_0$ and $S_1$ with the same colour, say $1$. Call $\psi$ the obtained colouring. Now, greedily extend $\psi$ to $G-x$ as in Remark~\ref{rq:greedyExtension}. We may assume that $\psi(N^+(x)) = [k-1]$ or $\psi(N^-(x)) = [k-1]$. 
 Since $d(x) \le 2(k-1)$, we have $1 \notin \psi(N(x)\cap X)$.
Set $\psi(x) = 1$. We may assume that $\psi$ is not a dicolouring of $G$. So we have an induced cycle $C$ containing $x$. $x$ has exactly two neighbours with colour $1$, one in $S_0$, the other in $S_1$. Hence $V(C) \cap S_0 \ne \varnothing$ and $V(C) \cap S_1 \ne \varnothing$. Since $G-X$ is disconnected, $V(C) \cap (X-x) \ne \varnothing$. Let $y$ be the last vertex of $V(C) \cap (X-x)$ to be coloured. Since the neighbours of $x$ in $C$ are not in $X$, the neighbours of $y$ in $V(C)$ were coloured when colouring $y$. Since we extended $\psi$ greedily, $\psi(y) \ne 1$, a contradiction.

\end{proof}

We need the following technical lemma on (undirected) bipartite graphs.
\begin{lem}[Lemma~3.6 in\cite{Stiebitz82}]\label{lem:sti6} Let $B$ be a bipartite graph with partite sets $S$ and $T$, such that $\abs T = \abs S + 1$ and $B$ contains a $2$-forest with respect to $S$. There exists $s \in S$ such that for every $t, t' \in N(s)$, $B$ contains a $2$-forest with respect to $S$ containing $st$ and $st'$.
\end{lem}

The next lemma is again a strong version of Theorem~\ref{thm:nbcompconnex_contr} in a particular case.

\begin{lem}\label{lem:2forest} 
Let $G$ be a connected digraph, $X \subseteq V(G)$, $n = \abs{\pi_0(G[X])}$ and $P = (P_0, ..., P_n)$ a partition of $\pi_0(G-X)$ such that:\begin{itemize}
  \item $\forall u \in X, d(u) \le 2(k-1)$.
  \item $B(G, X, P)$ contains a $2$-forest with respect to $\pi_0(G[X])$.
\end{itemize}
Then, for any $(k-1)$-dicolouring $\phi$ of $G-X$, there is a $(k-1)$-dicolouring $\psi$ of $G$ so that $\forall 0 \le i \le n, \exists \sigma \in \mathfrak S_{k-1}, \phi_{|\bigcup_{C \in P_i} C} = \sigma \circ \psi_{|\bigcup_{C \in P_i} C}$.\\

\end{lem}

\begin{proof}
We show the claim by induction on $\abs X$.

By Lemma~\ref{lem:case1}, we may assume $G[X]$ disconnected. 
Set $B = B(G,X,P)$. 
Let $\phi$ be a $(k-1)$-dicolouring of $G-X$. By Lemma~\ref{lem:sti6}, we have $S \in \pi_0(G[X])$ such that, for any $0 \le i \ne j \le n$ such that $SP_i, SP_j \in E(B)$, $B$ contains a $2$-forest with respect to $\pi_0(G[X])$ containing $SP_i$ and $SP_j$. Let $s$ be a non-separating vertex of $G[S]$. We distinguish two cases:

Assume first that $\abs{\{0 \le i \le n, \bid A(s, \bigcup_{C \in P_i} C) \ne \varnothing\} } \le 1$. 
Since $B$ contains a $2$-forest with respect to $\pi_0(G[X])$, $d_B(S) \geq 2$ and thus there is a vertex in $S \setminus s$ that has a neighbour in $X$. In particular, $\abs S \ge 2$. Since $S$ is connected, $d_{G-S+s}(s) < 2(k-1)$, so we can extend greedily $\phi$ to $G-(X-s)$.
Since $s$ is non-separating in $G[S]$, $\abs{\pi_0(G[X-s])} = \abs{\pi_0(G[X])}$. 
If $N(s) \subseteq S$, then $\{s\}$ is a connected component of $G-X + s$, and we set $P' = (P_0 + \{s\}, P_1, ..., P_n)$. 
Otherwise, let $C \in \pi_0(G-(X-s))$ such that $s \in V(C)$. Up to reindexing $P$, we may assume that $C-s \subset \bigcupl_{C' \in P_0} C'$ 
and set $P' = (\{C' \in P_0, C' \cap C = \varnothing\} + C, P_1, ..., P_n)$. Now, $B(G, X, P)$ is isomorphic to a spanning subdigraph of $B(G, X-s, P')$ and hence $B(G, X-s, P')$ contains a $2$-forest. We conclude by induction.

Assume now that $\abs{\{0 \le i \le n, \bid A(s, \bigcup_{C \in P_i} C) \ne \varnothing\}} \ge 2$.

Up to reindexing $P$, we may assume $\bid A(s, \bigcupl_{C \in P_0} C) \ne \varnothing$ and $\bid A(s, \bigcupl_{C \in P_1} C) \ne \varnothing$. 
Let $u_0 \in N(s) \cap \bigcupl_{C \in P_0} C$ and $u_1 \in N(s) \cap \bigcupl_{C \in P_1} C$ and $C_0, C_1 \in \pi_0(G-X)$ containing $u_0$ and $u_1$ respectively. 
By directional duality, we may assume $u_0 \in N^+(s)$.

 Up to permuting colours in $C_0$ and $C_1$, we may assume $\phi(u_0) = \phi(u_1)=1$. 
 Let $G' = G \cup u_0u_1 - S$ and $X' = X-S$. Note that $C_0+C_1$ is a connected component of $G'$. 
 We set $P' = (P_0 -C_0 + P_1 - C_1 + (C_0+C_1) , P_2, \dots, P_n)$. 
 Note that $\phi$ is a dicolouring of $G'-X'$ and $P'$ is a partition of $\pi_0(G'-X')$. 
 As $B(G', X', P') = B(G,X,P)-S/\{P_0, P_1\}$, the $2$-forest in $B(G, X, P)$ containing $SP_0$ and $SP_1$ yields a $2$-forest in $B(G', X', P')$. 
 Hence, by induction hypothesis, we may turn $\phi$ into a dicolouring $\psi$
 of $G'$ with the properties of the output of the theorem.
 
Note that $\psi$ is a dicolouring of $G-S$. 
We extend $\psi$ to $G-s$ as in remark~\ref{rq:greedyExtension}, and we may assume that $\psi(N^-(s)) = [k-1]$ and $\psi(N^+(s)) = [k-1]$.

 Set $\psi(s) = 1 = \psi(u_0) = \psi(u_1)$. 
 Since $\psi(N^+(s)) = [k-1]$ and $u_0 \in N^+(s)$, we have that $u_1 \in N^-(s)$.

 We may assume that there is a monochromatic induced cycle $R$ containing $s$ (otherwise we are done). 
 Observe that $s$ has exactly two neighbours with colour $1$, namely $u_0$ and $u_1$, so $R$ contains $u_1su_0$.
  Since $\psi$ is also a dicolouring of $G'$ and $u_0u_1 \in A(G')$, there is no monochromatic walk from $u_1$ to $u_0$ in $G-S$. So there is a vertex $y \in V(R) \cap (V(S)-s)$. Assume $y$ is the last vertex in $V(R) \cap (V(S)-s)$ to be coloured. Since the neighbours of $s$ in $R$ are not in $S$, the neighbours of $y$ in $R$ were coloured when colouring $y$. Since we extended $\psi$ greedily, $\psi(y) \ne 1$, a contradiction.

\end{proof}

We need a second technical lemma on (undirected) bipartite graphs before concluding.

\begin{lem}[Lemmas~3.4 and~3.5 in~\cite{Stiebitz82}]\label{lem:sti45} Let $B$ be a bipartite graph with partite sets $S$ and $T$ such that $\abs T \ge \abs S + 1$ and, for any $S' \in \m P(S)-\{\varnothing, S\}$, $\abs{\pi_0(B-S')} \le \abs{S'}$. Let $s \in S$ and $t \ne t' \in T$. Then $B$ contains a $2$-forest with respect to $S$ which contains $st$ and $st'$.
\end{lem}

\begin{proof}[Proof of Theorem~\ref{thm:nbcompconnex_contr}.]
We prove the result by induction on $\abs X$. 
By claim~\ref{lem:case1}, we may assume $G[X]$ disconnected. By induction hypothesis, we may assume that, for any $P \in \mathcal P(\pi_0(G[X]) - \{\varnothing, \pi_0(G[X])\}$, we have $\abs{\pi_0(G - \bigcupl_{C \in P} C)} \le \abs{P}$, for otherwise we can apply induction on $\bigcupl_{C \in P} V(C)$.
Let $P = (P_0, ..., P_{\abs{\pi_0(G[X])}})$ be a partition of $\pi_0(G-X)$. By Lemma~\ref{lem:sti45}, $B(G, X, P)$ has a $2$-forest. Since $\dic(G-X) \le k-1$, by Lemma~\ref{lem:2forest}, $\dic(G) \le k-1$.
\end{proof}

\section{List-dicolouring}\label{sec:list}

Let $G$ be a digraph.
A \emph{list assignment} of $G$ is a mapping $L: V(G) \to \mathcal{P}(C)$, where $C$ is a set of colours.
An \emph{$L$-dicolouring} of $G$ is a dicolouring $\phi$ of $G$ such that $\phi(v) \in L(v)$ for all $v\in V(G)$.
If $G$ admits an $L$-dicolouring, then it is \emph{$L$-dicolourable}. 
If $H$ is a subgraph of $G$, we abuse notations and write $L$ for the restriction of $L$ to $H$. 
Recall that, given a vertex $x$ of a digraph, $d_{max}(x) = \max(d^+(x), d^-(x))$ and $d_{min}(x) = \min(d^+(x), d^-(x))$. 

In~\cite{harutyunyan_gallais_2011}, Mohar and Harutyunyan proved the following, generalising a fundamental result of Gallai~\cite{Gal63a}. 

\begin{thm}[Theorem~2.1 in~\cite{harutyunyan_gallais_2011}]
Let $G$ be a connected digraph, and $L$ a list-assignment for $G$ such that $\abs{L(v)} \geq d_{max}(v)$ for every $v \in V(G)$. If $D$ is not $L$-dicolourable, then $d^+(v) = d^-(v)$ for every $v \in V(G)$ and every block of $G$ is a cycle, a symmetric odd cycle, or a complete digraph. 
\end{thm}

Observe that, in the above theorem, the blocks can not be arcs, so the output is a particular type of directed Gallai forest. Later on, Bang-Jensen et al. generalised the result of Mohar and Harutyunyan by proving Theorem~\ref{thm:directed_Gallai} that we restate here for convenience.

\begin{thm}[Bang-Jensen, Bellitto, Schweser and Stiebitz~\cite{bang2019haj}]\label{thm:bang2019haj}
If $G$ is a $k$-dicritical digraph, then the subdigraph induced by vertices of degree $2(k-1)$ is a directed Gallai forest. 
\end{thm}

Interestingly, contrary to the directed case, the undirected analogues of the two previous results both output an (undirected) Gallai forest, that is a graph whose blocks are odd (undirected) cycles or complete graphs. 

The goal of this section is to generalise the result of Bang-Jensen et al. by generalising a theorem proved by Thomassen~\cite{THOMASSEN199767} in the undirected case. 

\begin{thm}\label{thm:notLcolour}
Let $G$ be a connected digraph, $X \subseteq V(G)$ connected and $L$ a list-assignment of $G$ such that $G-X$ is $L$-dicolourable, $G$ is not $L$-dicolourable and $\forall x \in X, \abs{L(x)} \ge d_{\max}(x)$. 
Then $G[X]$ is a directed Gallai forest.
\end{thm}

The proof of Theorem~\ref{thm:notLcolour} is almost the same as the proof of Theorem~\ref{thm:bang2019haj}. 

The next proposition states some easy yet important facts that will be often used during the proof. 

\begin{prop}\label{prop:tek_list}
Let $G$ be a connected digraph, $X \subseteq V(G)$ connected and $L$ a list-assignment of $G$ such that $G-X$ is $L$-dicolourable, $G$ is not $L$-dicolourable and $\forall x \in X, \abs{L(x)} \ge d_{\max}(x)$. 

Then, for every $x \in X$, the following statements hold:
\begin{enumerate}
  \item\label{prop1} $\abs{L(x)} = d^+(x) = d^-(x)$,
  \item\label{prop2} $G-x$ is $L$-dicolourable. 
  \item\label{prop3} For every $L$-dicolouring of $G-x$, every colour of $L(x)$ appears in both $N^+(x)$ and $N^-(x)$. 
  \item\label{prop4} Given an $L$-dicolouring $\phi$ of $G-x$ and $y \in X \cap N(x)$, uncolouring $y$ and colouring $x$ with the colour of $y$ yields an $L$-dicolouring of $G-y$. 
\end{enumerate}
\end{prop}

\begin{proof}
Let $x \in X$. 

To prove~\ref{prop1}, it suffices to show that $\abs{L(x)} \le d_{\min}(x)$. We prove it for any $G$, $X$, $L$ and $x$ by induction on $\abs{V(G)}$. If $\abs{V(G)} \le 2$, the result is clear, so assume $\abs{V(G)} \ge 3$
Assume towards a contradiction that $\abs{L(x)} > d_{\min}(x)$.
Let $G' = G-x$. 
We can greedily extend any $L$-dicolouring of $G'$ to an $L$-dicolouring of $G$, so $G'$ is not $L$-dicolourable. 
Hence $G'$ has a connected component $C'$ that is not $L$-dicolourable. 
Since $G-X$ is $L$-dicolourable, $C' \cap X \ne \varnothing$. Furthermore, since $X$ is connected, we have $y \in C' \cap X \cap N(x)$. 
By the induction hypothesis applied to $G[C']$, $C' \cap X$ and $L$, we have $\abs{L(y)} = d_{G[C']}^+(y) = d_{G[C']}^-(y)$. 
By directional duality, we may assume $x \in N^+(y)$. Then: $d_{G[C']}^+(y) = \abs{L(y)} \ge d^+(y) \ge d_{G[C']}^+(y)+1$, 
a contradiction. This proves the first statement.

We now prove~\ref{prop2}. It suffices to prove that every connected component of $G-x$ is $L$-dicolourable. Let $C \in \pi_0(G-x)$. Let $D_1, ..., D_n$ be the connected components of $G[C \cap X]$. We prove by induction on $i \in \intZ{0, n}$ that $G[C-X+D_1+\dots+D_i]$ is $L$-dicolourable. Since $G-X$ is $L$-dicolourable, $G[C-X]$ is too. Now, let $i \in \intZ{0, n-1}$ and assume $G[C-X+D_1+\dots+D_i]$ $L$-dicolourable. Since $X$ is connected, we have $y \in D_{i+1} \cap N(x)$.
We have $\abs{L(y)} = d^+(y) = d^-(y) > d_{\min, G[C]}(y)$, so the first statement applied to $G[C-X+D_1+\dots+D_i]$, $D_{i+1}$, $L$ and $y$ yields that $G[C-X+D_1+\dots+D_i]$ is $L$-dicolourable, which concludes the proof.

Statement~\ref{prop3} follows easily from the fact that $G$ is not $L$-dicolourable. 
 
For the proof of~\ref{prop4}, assume (by symmetry) that $xy \in A(G)$. 
It follows from the third statement that, after uncolouring $y$, $x$ has no out-neighbour coloured $\phi(y)$, and thus giving colour $\phi(y)$ to $x$ does not create a monochromatic cycle. 
\end{proof}

In the rest of the proof, we will call the procedure that is described in Proposition~\ref{prop:tek_list}~\ref{prop4} \emph{shifting} the colour from $y$ to $x$, and sometimes write briefly $y \rightarrow x$. 
Moreover, given $G$, $X$ and $L$ as in the statement of Proposition~\ref{prop:tek_list}, a weak cycle $C=(v_1, a_1, v_2, \dots v_k, a_k, v_1)$ in $G[X]$ and an $L$-colouring of $G-v_1$, we can shift each vertex of $C$ one after another, starting with $v_k \rightarrow v_1 $ and get a new $L$-dicolouring of $G-v$. We say that we \emph{clockwise shift colours around $C$}, see Figure~\ref{fig_shift-example}. Starting with $v_2 \rightarrow v_1$, we say that we \emph{counter-clockwise shift colours around $C$}

 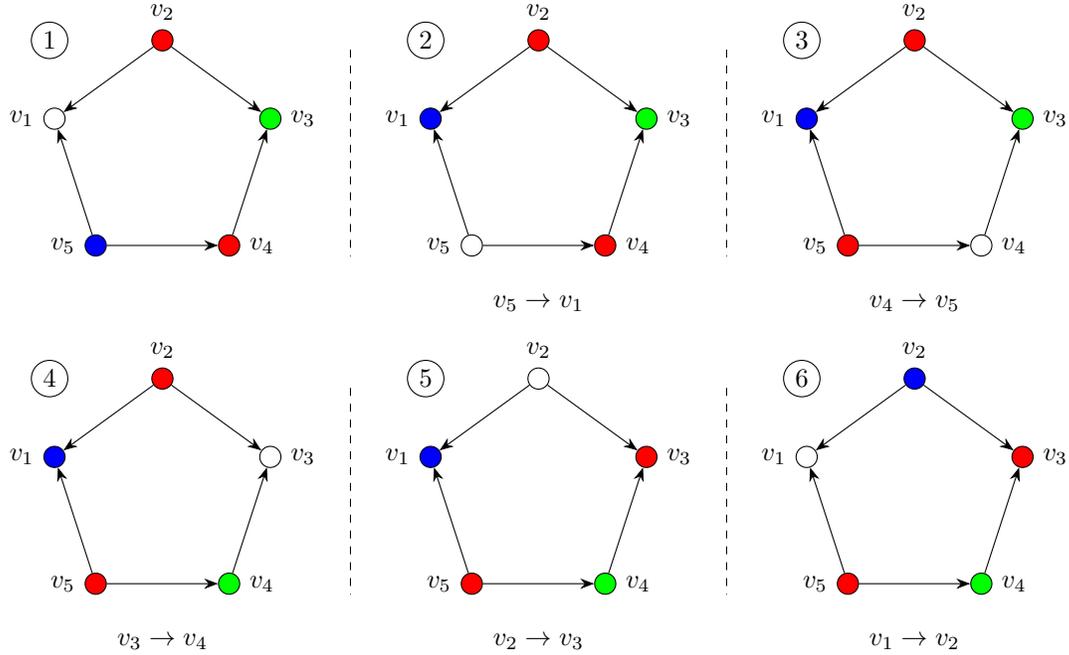
\begin{figure}[htbp]
\centering

\begin{tikzpicture}[>={[scale=1.1]Stealth}]
\node[draw=none,minimum size=3cm,regular polygon,regular polygon sides=5] (a) {};

\node[vertex, fill=red, label={above:$v_2$}] (a1) at (a.corner 1) {};
\node[vertex, label={left:$v_1$}] (a2) at (a.corner 2){};
\node[vertex, fill=blue, label={left:$v_5$}] (a3) at (a.corner 3){};
\node[vertex, fill=red, label={right:$v_4$}] (a4) at (a.corner 4){};
\node[vertex, fill=green, label={right:$v_3$}] (a5) at (a.corner 5){};
\node[circle, inner sep=2pt, draw=black] at (a1) [xshift=-1.5cm]{$1$};

\path[-]
(a1) edge [->] (a2)
(a2) edge [<-] (a3)
(a3) edge [->] (a4)
(a4) edge [->] (a5)
(a1) edge [->] (a5);

\node (h1) at (a1) [xshift=2.5cm]{};
\node (h2) at (h1) [yshift=-3cm]{};
\draw[dashed] (h1) -- (h2);

\begin{scope}[xshift=5cm]{};
\node[draw=none,minimum size=3cm,regular polygon,regular polygon sides=5] (a) {};

\node[vertex, fill=red, label={above:$v_2$}] (a1) at (a.corner 1) {};
\node[vertex, fill= blue, label={left:$v_1$}] (a2) at (a.corner 2){};
\node[vertex, label={left:$v_5$}] (a3) at (a.corner 3){};
\node[vertex, fill=red, label={right:$v_4$}] (a4) at (a.corner 4){};
\node[vertex, fill=green, label={right:$v_3$}] (a5) at (a.corner 5){};

\path[-]
(a1) edge [->] (a2)
(a2) edge [<-] (a3)
(a3) edge [->] (a4)
(a4) edge [->] (a5)
(a1) edge [->] (a5);

\node (h1) at (a1) [xshift=2.5cm]{};
\node (h2) at (h1) [yshift=-3cm]{};
\draw[dashed] (h1) -- (h2);
\node (l1) at (a1) [yshift=-3.5cm] {$v_5 \to v_1$};
\node[circle, inner sep=2pt, draw=black] at (a1) [xshift=-1.5cm]{$2$};
\end{scope}

\begin{scope}[xshift=10cm]{};
\node[draw=none,minimum size=3cm,regular polygon,regular polygon sides=5] (a) {};

\node[vertex, fill=red, label={above:$v_2$}] (a1) at (a.corner 1) {};
\node[vertex, fill= blue, label={left:$v_1$}] (a2) at (a.corner 2){};
\node[vertex, fill=red, label={left:$v_5$}] (a3) at (a.corner 3){};
\node[vertex, label={right:$v_4$}] (a4) at (a.corner 4){};
\node[vertex, fill=green, label={right:$v_3$}] (a5) at (a.corner 5){};

\path[-]
(a1) edge [->] (a2)
(a2) edge [<-] (a3)
(a3) edge [->] (a4)
(a4) edge [->] (a5)
(a1) edge [->] (a5);

\node (h1) at (a1) [xshift=2.5cm]{};
\node (h2) at (h1) [yshift=-3cm]{};

\node (l1) at (a1) [yshift=-3.5cm] {$v_4 \to v_5$};
\node[circle, inner sep=2pt, draw=black] at (a1) [xshift=-1.5cm]{$3$};
\end{scope}

\begin{scope}[yshift=-4.5cm]{};
\node[draw=none,minimum size=3cm,regular polygon,regular polygon sides=5] (a) {};

\node[vertex, fill=red, label={above:$v_2$}] (a1) at (a.corner 1) {};
\node[vertex, fill= blue, label={left:$v_1$}] (a2) at (a.corner 2){};
\node[vertex, fill=red, label={left:$v_5$}] (a3) at (a.corner 3){};
\node[vertex, fill=green, label={right:$v_4$}] (a4) at (a.corner 4){};
\node[vertex, label={right:$v_3$}] (a5) at (a.corner 5){};

\path[-]
(a1) edge [->] (a2)
(a2) edge [<-] (a3)
(a3) edge [->] (a4)
(a4) edge [->] (a5)
(a1) edge [->] (a5);

\node (h1) at (a1) [xshift=2.5cm]{};
\node (h2) at (h1) [yshift=-3cm]{};
\draw[dashed] (h1) -- (h2);
\node (l1) at (a1) [yshift=-3.5cm] {$v_3 \to v_4$};
\node[circle, inner sep=2pt, draw=black] at (a1) [xshift=-1.5cm]{$4$};
\end{scope}

\begin{scope}[yshift=-4.5cm, xshift=5cm]{};
\node[draw=none,minimum size=3cm,regular polygon,regular polygon sides=5] (a) {};

\node[vertex, label={above:$v_2$}] (a1) at (a.corner 1) {};
\node[vertex, fill= blue, label={left:$v_1$}] (a2) at (a.corner 2){};
\node[vertex, fill=red, label={left:$v_5$}] (a3) at (a.corner 3){};
\node[vertex, fill=green, label={right:$v_4$}] (a4) at (a.corner 4){};
\node[vertex, fill=red, label={right:$v_3$}] (a5) at (a.corner 5){};

\path[-]
(a1) edge [->] (a2)
(a2) edge [<-] (a3)
(a3) edge [->] (a4)
(a4) edge [->] (a5)
(a1) edge [->] (a5);

\node (h1) at (a1) [xshift=2.5cm]{};
\node (h2) at (h1) [yshift=-3cm]{};
\draw[dashed] (h1) -- (h2);
\node (l1) at (a1) [yshift=-3.5cm] {$v_2 \to v_3$};
\node[circle, inner sep=2pt, draw=black] at (a1) [xshift=-1.5cm]{$5$};
\end{scope}

\begin{scope}[yshift=-4.5cm, xshift=10cm]{};
\node[draw=none,minimum size=3cm,regular polygon,regular polygon sides=5] (a) {};

\node[vertex, fill=blue, label={above:$v_2$}] (a1) at (a.corner 1) {};
\node[vertex, label={left:$v_1$}] (a2) at (a.corner 2){};
\node[vertex, fill=red, label={left:$v_5$}] (a3) at (a.corner 3){};
\node[vertex, fill=green, label={right:$v_4$}] (a4) at (a.corner 4){};
\node[vertex, fill=red, label={right:$v_3$}] (a5) at (a.corner 5){};

\path[-]
(a1) edge [->] (a2)
(a2) edge [<-] (a3)
(a3) edge [->] (a4)
(a4) edge [->] (a5)
(a1) edge [->] (a5);

\node (h1) at (a1) [xshift=2.5cm]{};
\node (h2) at (h1) [yshift=-3cm]{};

\node (l1) at (a1) [yshift=-3.5cm] {$v_1 \to v_2$};
\node[circle, inner sep=2pt, draw=black] at (a1) [xshift=-1.5cm]{$6$};
\end{scope}

\end{tikzpicture}

\caption{The white vertex denotes the uncoloured vertex during the clockwise shifting around the weak cycle.}
\label{fig_shift-example}
\end{figure}

\begin{lem}\label{lem:notLcolourCycle} Let $G$ be a connected digraph, $X \subseteq V(G)$ connected and $L$ a list-assignment for $G$ such that $G-X$ is $L$-dicolourable, $G$ is not $L$-dicolourable and $\forall x \in X, \abs{L(x)} \ge d_{\max}(x)$. 
Let $C$ be a weak cycle in $G[X]$ of length $k \ge 3$ that is not a cycle. Then $V(C)$ is either a clique or induces an odd symmetric cycle.
\end{lem}

\begin{proof} Write $C = (v_1, a_1, v_2 \dots, v_k, a_k, v_1)$.
We prove the result by induction on $k$. 
All along the proof, subscripts are taken modulo $k$. In particular, $v_k$ and $v_1$ are considered to be consecutive vertices of $C$. 

\begin{claim}\label{clm:consec}
For every $i \in [k]$ and any $L$-dicolouring $\phi$ of $G-v_i$, no two consecutive vertices of $C$ receive the same colour. Moreover, $\phi(v_{i-1}) \ne \phi(v_{i+1})$. 
\end{claim}

\begin{proofclaim}
Let $i \in [k]$ and let $L$ be an $L$-dicolouring of $G-v_i$. 
Assume towards a contradiction that two consecutive vertices in $C$ have the same colour. 
Since $C$ is not a cycle of $G$, there exists $j \in [k]$ such that $v_{j-1}$ and $v_{j+1}$ are both in-neighbours of $v_j$ or both out-neighbours of $v_j$. 
We may shift colours around $C$ until $v_j$ is left uncoloured and $v_{j-1}$ and $v_{j+1}$ have the same colour, a contradiction to Proposition~\ref{prop:tek_list}~\ref{prop3}. 
Now, if $\phi(v_{i-1}) = \phi(v_{i+1})$, we can simply shift the colour from $v_{i-1}$ to $v_i$ and get a contradiction with the first fact. 
\end{proofclaim}

By Proposition~\ref{prop:tek_list}~\ref{prop2}, we have an $L$-dicolouring $\phi$ of $G-v_1$. \medskip

First suppose that $k$ is odd. Up to shifting colours and renaming the vertices, we may assume that $a_k = v_kv_1$ and $a_1 = v_1v_2$. We consider two cases.

Assume first that there is an arc $a \in A(G)$ between $v_1$ and $v_i$ for some $2 < i < k$. Let $C_0 = (v_1, a_1, v_2, \dots, v_i, a, v_1)$ and $C_1 = (v_1, a, v_i, a_i, ,v_{i+1},\dots v_k, a_k,v_1)$. 
One of $C_0$ and $C_1$ is not a cycle and hence, by induction, $v_1v_iv_1 \subseteq A(G)$. 
By symmetry, we may assume that $C_0$ is even and $C_1$ odd. 
Choosing the appropriate arc between $v_1$ and $v_i$ makes $C_0$ acyclic and hence, by induction, $V(C_0)$ is a clique. 
Similarly, $V(C_1)$ induces a symmetric cycle or is a clique. 
For $j \in \intZ{2, i-1}$, let $C_j = (v_1, v_jv_1, v_j, v_iv_j, v_i, a_i, v_{i+1}, \dots, v_k, a_k, v_1)$. 
Since $C_1$ is odd, $C_j$ is even, so by induction, $V(C_j)$ is a clique.
Hence $V(C)$ is a clique.

Now, suppose there is no arc between $v_1$ and $v_i$ for $i \in \intZ{3, k-1}$. 
By claim~\ref{clm:consec}, $\phi(v_k) \ne \phi(v_2)$. 

If the (unique) out-neighbour of $v_1$ with colour $\phi(v_k)$ is not $v_k$, then we shift colours clockwise around $C$ and get two out-neighbours of $v_1$ with the same colour, a contradiction to Proposition~\ref{prop:tek_list}~\ref{prop3}. 

Thus, $v_1v_k \in A(G)$.
Similarly, $v_2v_1 \in A(G)$. 
Hence, we have either $v_1v_2v_3 \subseteq A(G)$ or $v_3v_2v_1 \subseteq A(G)$, so we can repeat the argument and get a digon between $v_2$ and $v_3$. This way, we get that there is a digon between each pair of consecutive vertices of $C$ and thus $G[C]$ is a symmetric odd cycle.

\medskip 

Suppose now that $k$ is even.
Up to shifting colours and renaming the vertices, we may assume that $a_k = v_kv_1$ and $a_1 = v_2v_1$. 
By claim~\ref{clm:consec}, $\phi(v_k) \ne \phi(v_2)$ and $\abs{\{\phi(v_i), 2 \le i \le k\}} \ge 3$.

Let $3 \le j \le k-1$ such that $\phi(v_j) \notin \{\phi(v_2), \phi(v_{k-1})\}$. 
We shift colours around $C$ until $v_2$ is coloured $\phi(v_j)$. 
By Proposition~\ref{prop:tek_list}~\ref{prop3}, $\phi(v_2)$ and $\phi(v_k)$ still appear in the in-neighbourhood of $v_1$ and thus we have $3 \le i \le k-1$ such that $v_iv_1 \in A(G)$.

Assume first that $i$ is even. Then both $(v_1, a_1, v_2, \dots, v_i, v_iv_1, v_1)$ and $(v_1, v_iv_1, v_i, a_i, v_{i+1}, \dots, v_k, a_k, v_1)$ are even and are not a cycle, so by induction, $\{v_1,v_2 \dots, v_i\}$ and $\{v_i, v_{i+1}, \dots, v_k\}$ are cliques.
Hence 
$(v_1, v_1v_3, v_3, \dots, v_k, a_k, v_1)$ is odd, and $\{v_1, v_3, v_4 \dots, v_k\}$ does not induce a symmetric odd cycle (because $v_3$ and $v_k$ are adjacent). So, by induction, $\{v_1, v_3, v_4, \dots, v_k\}$ is a clique.  
The same holds for\\ $(v_1, a_1, v_2, \dots, v_{i-2}, v_{i-2}v_i, v_i, a_i, v_{i+1}, \dots, v_k, a_k, v_1)$, so $V(C)$ is a clique.

Assume now that $i$ is odd. So, $(v_1, a_1, v_2, \dots, v_i, v_iv_1, v_1)$ and $(v_1, v_iv_1, v_i, a_i, v_{i+1},\dots, v_k, a_k, v_1)$ are odd cycles and thus, by induction, each pair of consecutive vertices of $C$ induces a digon and $v_1v_iv_1 \subseteq G$. If $k=4$, then the argument of the paragraph following the assumption that $k$ is even finds a digon between $v_2$ and $v_4$. So we may assume $k \ge 6$.

Assume that both $\{v_1, v_2, \dots, v_i\}$ and $\{v_1, v_i, v_{i+1}, \dots, v_k\}$ induce a symmetric cycle. Since $k \geq 6$, one of $(v_1, a_1, v_2, \dots, v_i, v_iv_1, v_1)$ and $(v_1, v_iv_1, v_i, a_i, v_{i+1}, \dots, v_k, a_k, v_1)$ has length at least $5$. Assume without loss of generality that it is $(v_1, a_1, v_2, \dots, v_i, v_iv_1, v_1)$ (so $i \geq 5$). 
Counter-clockwise shifting colours around $(v_1, a_1, v_2, \dots, v_i, v_iv_1, v_1)$, and noticing that in the new $L$-dicolouring of $G-v_1$, the in-neighbours of $v_1$ have the same colours as in the previous one, we get that $\phi(v_{3}) = \phi(v_i)$. 
Now, counter-clockwise shifting colours (of $\phi$) around $(v_1, v_iv_1, v_i, a_i, v_{i+1}, \dots, v_k, a_k, v_1)$, the same argument yields $\phi(v_3) = \phi(v_{i+1})$. So $\phi(v_i) = \phi(v_{i+1})$, a contradiction to claim~\ref{clm:consec}.

Hence, we may assume without loss of generality that $\{v_1, v_2, \dots, v_i\}$ does not induce a symmetric odd cycle. 
In particular $i \geq 5$. 
By induction, $\{v_1, v_2, \dots, v_i\}$ induces a clique. 
By applying induction to $(v_1, v_3v_1, v_3, a_3, v_4, \dots, v_k, a_k, v_1)$, we get that $V(C)-v_2$ is a clique. 
Finally, by applying induction on $(v_1, a_1, v_2, v_4v_2, v_4, \dots, v_k, a_k, v_1)$, we get that $V(C)-v_3$ is a clique and thus that $V(C)$ is a clique.
\end{proof}

\begin{proof}[Proof of Theorem~\ref{thm:notLcolour}]
Let $B$ be a block of $G[X]$. 
If $\abs B \le 3$, $B$ is either an simple arc, a digon, a $\vec C_3$, or a $\bid K_3$ by Lemma~\ref{lem:notLcolourCycle}. 
So we assume $\abs{V(B)} \ge 4$. 
By~Lemma~\ref{lem:notLcolourCycle}, we may assume that $B$ is not a cycle. 
So there are two vertices in $V(B)$ 
linked by three internally vertex-disjoint weak walks. 
Call $P_0, P_1, P_2$ these three weak walks. Two of these walks form a weak cycle that is not a cycle. Hence by Lemma~\ref{lem:notLcolourCycle}, they form a symmetric cycle.
Two of $P_0$, $P_1$ and $P_2$, say $P_0$ and $P_1$, form a weak cycle $C$ of even length. One of $P_0$ and $P_1$ is symmetric, so up to choosing the arcs in it, $C$ is not a cycle.
By Lemma~\ref{lem:notLcolourCycle}, $V(C)$ is a clique and observe that $\abs{V(C)} \geq 4$. 
Let $R$ be a maximal clique containing $V(C)$. 
We may assume $R \ne V(B)$. 
Let $v \in V(B)-R$. 
Since $B$ is a block, there are two weak walks $P$ and $Q$ from $v$ to $R$ whose only common vertex is $v$. Let $p$ and $q$ their respective end-vertices in $R$.
Let $w \in R$. We have $z \in R-p-q-w$. One of $vPpwqQv$ or 
$vPpwzqQv$ is odd, and both can be chosen undirected and none of them is an induced cycle (because $p$ and $q$ are adjacent). Hence, by Lemma~\ref{lem:notLcolourCycle}, the vertices of one of them induce a clique, and thus $w$ is linked by digon to $V(P) \cup V(Q)$. So $R \cup V(P) \cup V(Q)$ is a clique, a contradiction to the maximality of $R$. 
\end{proof}

\section{Conclusion}\label{sec:furtherworks}

For $k \ge 4$,  let $\m F_k = \{\bid C_5(\bid K_1, \bid K_{a_1}, \bid K_{a_2}, \bid K_{b_2},\bid K_{b_1}) \mid a_1 + a_2 = b_1 + b_2= k-1, a_2 + b_2 = k-1\}$. Recall that we identify (undirected) graphs with symmetric digraphs. 
Kostochka and Stiebitz~\cite{KS99} proved that, for $k \geq 4$, $k$-critical graphs have excess at least $2(k-3)$ except for $K_k$ and graphs in $\m F_k$. 
It is natural to wonder if such a characterisation  exists for digraphs. Anyway, $k$-dicritical digraphs being more complicated that their undirected counter part, we doubt it.

Our bound on the minimal number of arcs in a $k$-dicritical digraph on $n$ vertices (Theorem~\ref{thm_intro:KY}) is clearly not tight. Kostochka and Stiebitz conjectured\cite{kostochka2020minimum} that $k$-dicritical digraphs on at least $k+1$ vertices with minimum density are symmetric, i.e. are the same as in the case of undirected graphs. It is to be noted that in an other breakthrough result, 
Kostochka and Yancey characterised the $k$-critical graphs that are tight for the bound of Theorem~\ref{thm:KY}.

\subsubsection*{Acknowledgement}
We are tankful to Clément Rambaud for fruitful discussions, and to Bang-Jensen, Bellito, Stiebitz and Schweser for Figure~\ref{fig_shift-example}. We are also thankful to two anonymous reviewers for their suggestions that simplified one of the proofs. 

This research was partially supported by ANR project DAGDigDec (JCJC)   ANR-21-CE48-0012 and by the group Casino/ENS Chair on Algorithmics and Machine Learning.

\bibliographystyle{alpha}

\newcommand{\etalchar}[1]{$^{#1}$}

\end{document}